\documentclass[11pt]{article}

    \usepackage{amsthm}
    \usepackage{amsmath}
    \usepackage{amssymb}
    \usepackage[margin=1in]{geometry}
    \usepackage{enumerate}
    \usepackage{hyperref}
    \usepackage{mathrsfs}
    \usepackage{color}
    \usepackage{bm}
    \usepackage[title]{appendix}
    \usepackage{ulem}
    \usepackage{prettyref}
    \usepackage{cancel}

\usepackage[bibstyle=alphabetic,citestyle=alphabetic-verb]{biblatex}

\title{Sparsity of Fourier mass of passively advected scalars in the Batchelor regime}
\author{Alex Blumenthal\footnote{This material is based upon work supported by the National Science Foundation under grant
nos. DMS-2009431 and DMS-2237360.} \and Manh Khang Huynh}
\date{\today}

    \newtheorem{theorem}{Theorem}[section]
    
    \newtheorem{corollary}[theorem]{Corollary}
    \newtheorem{lemma}[theorem]{Lemma}
    
    \newtheorem{proposition}[theorem]{Proposition}

    \theoremstyle{definition}
    
    \newtheorem{definition}[theorem]{Definition}

    \newtheorem{remark}[theorem]{Remark}
    
    \newtheorem{assumption}[theorem]{Assumption}
    \newtheorem{claim}[theorem]{Claim}
    \newtheorem{example}[theorem]{Example}

\begin{document}


\newcommand{\A}{\mathbb A}
\newcommand{\C}{\mathbb{C}}
\newcommand{\D}{\mathbb{D}}
\newcommand{\E}{\mathbb{E}}
\newcommand{\F}{\mathbb{F}}
\newcommand{\N}{\mathbb{N}}
\renewcommand{\P}{\mathbb{P}}
\newcommand{\R}{\mathbb{R}}
\newcommand{\X}{\mathbb{X}}
\newcommand{\Z}{\mathbb{Z}}


\newcommand{\Uc}{\mathcal{U}}
\newcommand{\Bc}{\mathcal{B}}
\newcommand{\Fc}{\mathcal{F}}
\newcommand{\Hc}{\mathcal{H}}
\newcommand{\Lc}{\mathcal{L}}
\newcommand{\Rc}{\mathcal{R}}
\newcommand{\Gc}{\mathcal{G}}
\newcommand{\Ac}{\mathcal{A}}
\newcommand{\Sc}{\mathcal{S}}
\renewcommand{\Hc}{\mathcal{H}}
\newcommand{\Oc}{\mathcal{O}}
\newcommand{\Wc}{\mathcal{W}}
\newcommand{\Nc}{\mathcal{N}}
\newcommand{\Mc}{\mathcal M}
\newcommand{\Ec}{\mathcal E}


\newcommand{\Bs}{\mathscr{B}}
\newcommand{\As}{\mathscr{A}}


\renewcommand{\a}{\alpha}
\renewcommand{\b}{\beta}
\newcommand{\g}{\gamma}
\renewcommand{\d}{\delta}
\newcommand{\e}{\epsilon}
\renewcommand{\l}{\lambda}
\renewcommand{\L}{\Lambda}


\newcommand{\lie}{\text{Lie}}
\newcommand{\dlangle}{\langle \langle}
\newcommand{\drangle}{\rangle \rangle}
\newcommand{\pd}{\partial}
\newcommand{\len}{\text{Len}}
\newcommand{\diam}{\operatorname{diam}}
\newcommand{\graph}{\operatorname{graph}}
\newcommand{\lip}{\operatorname{Lip}}
\newcommand{\Id}{\operatorname{Id}}
\newcommand{\dist}{\operatorname{dist}}
\newcommand{\ds}{/ \! /}
\newcommand{\codim}{\operatorname{codim}}
\newcommand{\Slope}{\operatorname{Slope}}
\newcommand{\Gap}{\operatorname{Gap}}
\newcommand{\Spec}{\operatorname{Spec}}

\newcommand{\T}{\mathbb T}
\newcommand{\Cc}{\mathcal C}
\newcommand{\Leb}{\operatorname{Leb}}
\newcommand{\Lip}{\operatorname{Lip}}
\renewcommand{\graph}{\operatorname{graph}}

\newcommand{\Len}{\operatorname{Len}}
\newcommand{\Pc}{\mathcal P}

\newcommand{\Bor}{\operatorname{Bor}}
\newcommand{\supp}{\operatorname{Supp}}
\newcommand{\Gr}{\operatorname{Gr}}

\newcommand{\Var}{\operatorname{Var}}

\newcommand{\Span}{\operatorname{Span}}

\newcommand{\PP}{\mathbf{P}}
\newcommand{\EE}{\mathbf{E}}
\newcommand{\bVar}{\mathbf{Var}}

\newcommand{\Reynolds}{\operatorname{Re}}

\maketitle

\begin{abstract}
    In 1959, Batchelor gave a prediction for the power spectral density of a passive scalar advected by an incompressible fluid exhibiting shear-straining, a mechanism for the creation of small scales in the scalar \cite{batchelor1959small}. Recently, a `cumulative' version of this law, summing over Fourier modes below a given wavenumber $N$, was given for a broad class of passive scalars under incompressible advection, including by solutions to the stochastic Navier-Stokes equations \cite{bedrossianBatchelorSpectrumPassive2020}. 
    
    This paper addresses to what extent Fourier mass of such passive scalars truly saturates the predicted power law scaling due to Batchelor. Via discrete-time pulsed-diffusion models of the advection-reaction equations, we exhibit situations compatible with the cumulative law but for which the distribution of Fourier mass among wavenumbers $|k| \leq N$ is relatively \emph{sparse} and much smaller than a `mode-wise' version of Batchelor's original prediction. In the same situations we also establish an `exponential radial shell' version of Batchelor's laws via a novel application of the method of spectral distributions.


    
\end{abstract}
  
\section{Introduction}\label{sec:intro}

    Let $\Omega = \T^d \cong [0,1)^d$ be the periodic box and assume $g : [0,\infty) \times \Omega \to \R$ is a solution to the following advection-diffusion equation: 
    \begin{align}\label{eq:PSAintro} \begin{cases}
        \partial_t g + u \cdot \nabla g = \kappa \Delta g + S  \\
        g(0, x) = g_0(x) 
    \end{cases}\end{align}
    Here, $u : [0,\infty) \times \Omega \to \R^d$ is a smooth, divergence free vector field (i.e., $\nabla \cdot u \equiv 0$); $S = S(t, x)$ is a time-varying, spatially-smooth source term with $\int S(t, x) dx \equiv 0$; $g_0 =  g_0(x) \in L^2(\T^d)$ is a fixed initial condition; and $\kappa > 0$ is a fixed real number. Throughout, $\nabla$ and $\Delta$ refer to the spatial gradient and Laplacian, respectively. Below and throughout, for $h : \T^d \to \R$ we write $\Fc[h](k) = \hat h(k)$ for the Fourier coefficient of $h$ at wavenumber $k = (k^1, \dots, k^d) \in \Z^d$ in the usual Fourier series representation $h = \sum_{k \in \Z^d} \hat h(k) e_k$, where $e_k(x) := e^{i 2\pi (k \cdot x)}, x \in \T^d,$ is the standard Fourier basis on $\T^d$.  For classical solutions $g$ of \eqref{eq:PSAintro} it holds that $\int g(t, x) dx$ is constant in time, and so without loss we will assume throughout that $\int g_0(x) dx = 0$, hence $\int g(t, x) dx \equiv 0$ for all $t \geq 0$.

    Equation \eqref{eq:PSAintro} models the time evolution of a passive scalar in a fluid moving according to the velocity field $u(t, x)$. When, e.g., $g(t, x)$ is the concentration of a solute in the fluid, the $\kappa \Delta$ term models molecular diffusion, and $\kappa$ itself is interpreted as a \emph{diffusivity} parameter. 
    This paper concerns the distribution of the Fourier mass of a passively advected scalar $g(t, x)$ evolving according to \eqref{eq:PSAintro} at long times $t$ and for small values of the diffusivity $\kappa$, a situation referred to as \emph{passive scalar turbulence}, analogous to other turbulent regimes such as hydrodynamic turbulence in the limit of vanishing viscosity in the Navier-Stokes equation. 
    
    Of interest in this paper is a special case of the so-called \emph{Batchelor regime} of passive scalar turbulence, which can be roughly stated as follows, using terms which will be elaborated on as we progress: 
    \begin{itemize}
        \item[(a)] the velocity field $u(t, x)$ is stationary in time $t$ and governed by some fixed, $\kappa$-independent statistical law, e.g., statistically stationary solutions to the Navier-Stokes equations with $d = 2$ at fixed Reynolds number $\Reynolds$; 
        \item[(b)] the passive tracer (Lagrangian) flow $\varphi^t : \T^d \to \T^d$ generated by $u(t, x)$, defined from 
        \begin{align}\label{eq:defnLagFlowIntro}\frac{d}{dt} \varphi^t(x) = u(t, \varphi^t(x)) \, , \end{align} 
        exhibits \emph{shear-straining} across the domain $\Omega$, i.e., the spatial Jacobians $|D \varphi^t(x)|$ grow exponentially in time, leading to strong expansion/contraction of tangent vectors in various directions; and 
        \item[(c)] the source term $S(t, x)$ is also stationary in time, $\kappa$-independent, and localized to large scales (low Fourier modes). 
    \end{itemize}
     Under these assumptions, there are three major competing effects in \eqref{eq:PSAintro}: 
    \begin{itemize}
        \item[(i)] advection by the velocity field $u$, which transmits $L^2$ mass of $g$ from coarse to fine spatial scales under assumptions (a) and (b); 
        \item[(ii)] diffusivity from the $\kappa \Delta$ term, which dampens features of $g$ at scales $\approx \sqrt{\kappa}$ when $\kappa > 0$ is small; and 
        \item[(iii)] the source term $S$, which introduces $L^2$ mass to $g$ at some fixed spatial scale by assumption (c). 
    \end{itemize}
    The combination of these three effects produces statistically stationary behavior for the scalar $g$ at long times: a nonequilibrium steady state balancing all three effects and giving rise to a fat tail of Fourier mass at wavenumbers $k \in \Z^d$, predicted to obey \emph{Batchelor's law}
    \begin{align} \label{eq:modewiseIntro} |\hat g(k)|^2 \approx |k|^{-d} \qquad \text{ for all } \qquad  k \in \Z^d, \, \, 1 \ll |k| \ll \kappa^{-1/2} \,. \end{align}
    The formation of small scales in passive scalar turbulence, outlined above, is analogous to the energy cascade in isotropic hydrodynamic turbulence. In this vein, the power spectral law of Batchelor is an analogue of the Kolmogorov $-5/3$ law for the power spectral density of a turbulent velocity field. 

    In the series of papers \parencite{bedrossianLagrangianChaosScalar2018,bedrossianAlmostsureExponentialMixing2019,bedrossianAlmostsureEnhancedDissipation2021}, it is shown that conditions (a) and (b) are met for a broad class of spatially Sobolev-regular velocity fields $u(t, x)$, including solutions to the Navier-Stokes equations with suitable white-in-time, spatially Sobolev random forcing. Finally, the following \emph{cumulative} version of Batchelor's law was shown in \cite{bedrossianBatchelorSpectrumPassive2020}: 
    \begin{align}\label{eq:cumLawIntro} \E \, \| \Pi_{\leq N} g\|_{L^2}^2 \approx \log N \qquad \text{ for } \qquad  1 \ll N \ll \kappa^{-1/2} \,, 
    \end{align}
    where $\Pi_{\leq N}$ is a Galerkin truncation to Fourier modes $|k| \leq N$, and where $\E$ averages over statistically stationary passive scalars $g$. 

    We refer to \eqref{eq:cumLawIntro} as a \emph{cumulative law} for the Fourier mass of $g$. This law is compatible with but strictly weaker than the ``mode-wise'' power law \eqref{eq:modewiseIntro} originally predicted by Batchelor. In all, the results of \cite{bedrossianBatchelorSpectrumPassive2020} provide substantial evidence that in this special case of the Batchelor regime, the cumulative law \eqref{eq:cumLawIntro} is, in fact, a universal one, whether or not the original prediction \eqref{eq:modewiseIntro} holds. 
   
    This paper concerns the validity of the ``mode-wise'' version \eqref{eq:modewiseIntro} of Batchelor's law, and is an initial study into the extent to which it is universal among incompressible fluid advection models. In what follows, we establish, using pulsed-diffusion models, a collection of mechanisms compatible with the cumulative law \eqref{eq:cumLawIntro} but for which the \emph{mode-wise law} \eqref{eq:modewiseIntro} \emph{fails}. In particular, this paper provides an accessible explanation for the power spectrum laws via harmonic analysis,  mirroring the analysis in \cite{bedrossianBatchelorSpectrumPassive2020} at each step. 

\subsection{Statement of results}

In this paper we consider \emph{pulsed-diffusion} models, discrete-time analogues of the advection-diffusion equations \eqref{eq:PSAintro}. 
Let $f : \T^d \to \T^d$ be a volume preserving smooth diffeomorphism, which we will take to model the advection term in \eqref{eq:PSAintro} (c.f. (i) above). The \emph{transfer operator} $\Lc_f$ of $f$ is defined as follows: given a measurable function $\varphi : \T^d \to \R$, we define $\Lc_f [\varphi] : \T^d \to \R$ by
\[\Lc_f [\varphi](x) = \varphi \circ f^{-1}(x) \,. \]
Observe that if $f = \varphi^1$, the time-1 Lagrangian flow as in \eqref{eq:defnLagFlowIntro}, then $\Lc_f[g_0](x) = g(1, x)$ is the time-1 solution to \eqref{eq:PSAintro} with $\kappa = 0, S \equiv 0$. 

Advection and diffusion (c.f. (i) and (ii) above) will be modelled jointly by compositions of the operator
\[\Lc_{f, \kappa} := e^{\kappa \Delta} \circ \Lc_f \, , \]
where $\kappa > 0$ is a fixed diffusivity parameter; here, for $t > 0$ we write $e^{t \Delta}$ for the time-$t$ heat semigroup. 
Finally, the full passive scalar evolution $(g(t, \cdot))$ in \eqref{eq:PSAintro} will be modelled by the sequence $(g_n)$ of scalars $g_n : \T^d \to \C$ defined inductively via 
\begin{align}\label{eq:defnPulseDiffIntro}
    g_{n+1} = \Lc_{f, \kappa} [g_n] + S_{n + 1} \, ,
\end{align}
where for $n \geq 1$ we set the time-$n$ source term $S_n$ to be 
\[S_n = \omega_n b \, , \qquad b = e_{k_0} \,.  \]
Here, $(\omega_n)$ is an IID sequence of standard normal Gaussian random variables, and $k_0 \in \Z^d$ is a fixed wavenumber; the choice $b = e_{k_0}$ is taken for simplicity. 

The following prototypical example motivates much of what follows. 
\begin{example}[Arnold CAT map]
    Let $d = 2$ and 
    \[A = \begin{pmatrix}
    2 & 1 \\ 1 & 1
\end{pmatrix} \, . \]
Define $f_A : \T^2 \to \T^2$ as $f(x) = A x$ mod 1. Here, $x \in \T^2$ is parametrized by a vector in $[0,1)^2$, and $A x$ mod 1 refers to the vector in $[0,1)^2$ obtained by taking mod 1 separately of each coordinate of $A x$. The resulting map $f_A$ is smooth and preserves volume (since $\det A \equiv 1$). 
\end{example}

It is straightforward to check that the $(g_n)$ form a Markov chain on $L^2(\T^d)$, and that $(g_n)$ admits a unique stationary measure $\mu_{\kappa}$ supported in $L^2(\T^d)$ and distributed like 
    \begin{align} \label{eq:statLawIntro1}\sum_{n = 0}^\infty \eta_n \Lc_{f_A, \kappa}^n [e_{k_0}] = \eta_0 e_{k_0} + \sum_{n = 1}^\infty \eta_n e^{- \kappa C (|k_1|^2 + \dots + |k_n|^2)} e_{k_{n}} \, , \end{align}
    where $k_n := (A^{-T})^{n} k_0$, $C$ is an unimportant constant, and $(\eta_n)$ is a (separate) IID sequence of standard normal random variables.     

Observe that the matrix $A$ admits two distinct real eigenvalues $\lambda > 1 > \lambda^{-1}$ with corresponding eigenvectors $v^{u}$ and $v^{s}$, where $|v^{u/s}| = 1$. As one can check, at least one coordinate entry of each of $v^u, v^s$ is irrational, and so it follows that $v^u$ and $k_0 \in \Z^d$ cannot be collinear. Consequently, $|k_n| \approx \lambda^n$ at large $n$. 

The following is now immediate. Let $g_{\kappa}$ be distributed like $\mu_{\kappa}$ (i.e., a random function in $L^2(\T^d)$ with law given from the RHS of \eqref{eq:statLawIntro1}). 
\begin{itemize}
    \item[(A)]     An analogue of the cumulative law \eqref{eq:cumLawIntro} is valid for $g_\kappa$. 
    \item[(B)] The Fourier mass of $g_\kappa$ is supported in the sequence of exponentially-spaced modes $k_1, k_2, \dots$, and this sequence of modes accumulates onto the line $\Span\{v^s\}$, the span of the lower eigenvector $v^s$. In particular, the mode-wise law \eqref{eq:modewiseIntro} is badly false for $g_\kappa$ distributed like $\mu_\kappa$.
\end{itemize}

While the cumulative law \eqref{eq:cumLawIntro} is valid, paragraph (B) seems to cast doubt on the universality of the stronger mode-wise law, since the actual distribution of Fourier mass is much `sparser' than what is predicted by \eqref{eq:modewiseIntro}. On the other hand, the Arnold CAT map $f_A$ is extremely special, in that it transmits pure Fourier modes to pure Fourier modes: it is unrealistic to expect the Fourier matrix of a transfer operator $\Lc_{f_A}$ to be so sparse for a `general' system. 

Towards a broader understanding of the validity of \eqref{eq:modewiseIntro}, it is natural to consider \emph{to what extent the situation in paragraph} (B) \emph{persists under small changes to the map $f_A$}. This is the subject of the main results of this paper, simplified versions of which are stated below. 

Below, $f$ is a volume-preserving, $C^\infty$ perturbation of the CAT map $f_A$. 
The following versions of paragraph (A) is valid in this setting.
\begin{proposition}[Informal]\label{prop:cumLawPerturbedIntro}
    The following hold for all $f$ which are $C^\infty$ close to $f_A$. 
    \begin{itemize}
        \item[(i)] For any $\kappa > 0$ it holds that $(g_n)$ admits a unique stationary measure $\mu_{f, \kappa}$  (with $g_{f, \kappa}$ being the corresponding statistically stationary scalar) supported in $L^2(\T^2)$; 
        \item[(ii)] The cumulative law \eqref{eq:cumLawIntro} holds. 
    \end{itemize}
\end{proposition}
A general sufficient condition for Batchelor's law in the pulsed-diffusion setting is given in Section \ref{sec:background}, essentially following \cite{bedrossianBatchelorSpectrumPassive2020}, while Proposition \ref{prop:cumLawPerturbedIntro} specifically for peturbations of the CAT map follows from the application of this criterion - see Section \ref{subsec:catMap3} for details. 

We now turn to the failure of the mode-wise law as in paragraph (B). We begin by addressing the Fourier support of $g_{f, \kappa}$ away from $\Span\{ v^s\}$. 

\begin{theorem}[Informal]\label{thm:sectorialIntro}
    Let $p>0$. There exists $\vartheta \in (0,\pi/2)$ such that for all $f$ $C^\infty$-close to $f_A$ and for all $\kappa>0$ sufficiently small, there holds
 
    \[\angle(k, \Span\{v^s\}) \geq \vartheta \quad \text{ implies } \quad \E |\hat g_{a,\kappa}(k)|^2 \lesssim \frac{1}{|k|^{p}}
    \]
  for all wavenumbers $k$. 

\end{theorem}
Theorem \ref{thm:sectorialIntro} is proved in Section \ref{sec:sectorial} using the technique of anisotropic spaces, a tool used for studying the spectral theory of transfer operators; see citations in Section \ref{subsec:references}. 

Taking $p > 2$ implies a violation of the `mode-wise' law \eqref{eq:modewiseIntro} for all wavenumbers $k$ sufficiently far from $v^s$. 
On the other hand, the `sector-wise' sparsity of Fourier mass predicted in Theorem \ref{thm:sectorialIntro} is compatible, at least, with a version of \eqref{eq:modewiseIntro} averaged over radial shells, e.g., 
\begin{align}\label{eq:radialLaw}
    \sum_{|k| \in [R-2, R + 2]} |\hat g(k)|^2 \approx \frac{1}{R} \quad \text{ for all } \quad 1 \ll R \ll \kappa^{-1/2} 
\end{align}
Analogues of the above relation are, indeed, often what is measured in  experiments and simulations in physics literature -- see references in Section \ref{subsec:references} below. 

Our next results address the breakdown of \eqref{eq:radialLaw} when $f$ is close to the CAT map $f_A$ by estimating, directly, the localization of Fourier mass to the sequence of `pulses' centered on $k_n := (A^{-T})^n k_0$. 

\begin{theorem}[Informal]\label{thm:pulseLocalizationIntro}
    Let $\epsilon, \kappa > 0$ be sufficiently small. Then, if $d_{C^\infty}(f, f_A) < \epsilon$, it holds that 
    \begin{gather}\label{eq:pulseControlIntro}
    \E \sum_{\substack{k \notin \{ k_\ell\} \\ |k| \leq R_{\rm crit}}} |\hat g_{f, \kappa}(k)|^2  \lesssim \left( \max\{ \kappa, \epsilon\} \right)^{\delta} \, , 
    \end{gather}
    where
    \[R_{\rm crit} := \big(\max \{ \kappa, \epsilon\} \big)^{-\zeta} \,. \]
  \end{theorem}

  From the cumulative law estimate \[ \E \| \Pi_{\leq R_{\rm crit}} g_{f, \kappa}(k)\|^2_{L^2} \approx \log R_{\rm crit} \gg 1 \] from Theorem \ref{thm:cumLaw2} and the exponential spacing of the `pulse' wavenumbers $|k_n| \approx \lambda^n$, it follows that equation \eqref{eq:pulseControlIntro} is badly incompatible with any version of the radial law \eqref{eq:radialLaw} summing over fixed-width bands of wavenumbers. This control is only valid, however, for wavenumbers below the critical threshold $R_{\rm crit}$. For further discussion, see Remark \ref{rmk:criticalTimescale4}. 
  
  Our last result, stated below, affirms that a version of the radial law is valid when summing over exponentially-growing bands of radii. 

  \begin{theorem}[Informal]\label{thm:expShellLawIntro}
    Let $\epsilon, \kappa > 0$ be sufficiently small and assume $d_{C^\infty}(f, f_A) < \epsilon$. Let $L > 1$. Then, 
    
    \begin{align}\label{eq:exponentialShellEstIntro}
       \E \| \Pi_{[L^\ell, L^{\ell + 1}]} g_{f, \kappa}\|^2_{L^2} = \frac{\log L}{\log \lambda} + O(1) 
    \end{align}
    for all $\ell \leq \ell_{\rm crit}$, where 
    \[\ell_{\rm crit} \approx \frac{1}{\log L} | \log \max\{ \epsilon, \kappa\}| \]
    Here, $O(1)$ refers to a constant term independent of $\ell, L, \kappa, \epsilon$, and $\Pi_{[L^\ell, L^{\ell + 1}]}$ to the projection onto Fourier modes $k$ with $L^\ell \leq |k| \leq L^{\ell + 1}$. 
  \end{theorem}
  
  We note that when $L \gg \lambda$, the value $\frac{\log L}{\log \lambda}$ is an approximation for the number of pulses $k_n$ located in the exponential band $[L^\ell, L^{\ell + 1}]$. This accounts for the form of the RHS of \eqref{eq:exponentialShellEstIntro}. For additional discussion on the validity of an analogue of \eqref{eq:exponentialShellEstIntro} for more general systems, see the discussion in Section \ref{subsec:discussion2}. 

  The proofs of Theorems \ref{thm:pulseLocalizationIntro} and \ref{thm:expShellLawIntro} involve the use of \emph{spectral distributions}, a tool from signal and audio processing \cite{oppenheimSignalsSystemsInference2016}.

\subsection{Discussion and comments on previous work}\label{subsec:references}

The study of the power spectrum of passive scalars has a long history-- see, e.g., Batchelor's original paper \cite{batchelor1959small} as well as the influential early works \parencite{gibson1968fine, kraichnan1968small}, each of which provide a plausible (nonrigorous) argument for the mode-wise \eqref{eq:modewiseIntro} and radial laws \eqref{eq:radialLaw} given above. For more contemporary accounts, see, e.g., the surveys \parencite{warhaft2000passive, shraiman2000scalar, donzis2010batchelor} and references therein. We acknowledge that many aspects of passive scalar turbulence are completely unaddressed by this manuscript, including: ideal turbulence \parencite{eyink2024onsager, bedrossianBatchelorSpectrumPassive2020}; the Obukhov-Corrsin correction \parencite{obukhov1949structure,corrsin1951spectrum} when $\Reynolds$ is large; the Batchelor-Howell-Townsend spectrum for passive scalar power spectral density at spatial ranges where diffusivity predominates \cite{BHT1959} (c.f. work of Jolly and Wirosoetisno \parencite{jolly2023batchelor, jolly2023batchelor2, jolly2020tracer}); and intermittency effects in passive scalar turbulence \cite[Section 4]{warhaft2000passive}. 

Since \cite{bedrossianBatchelorSpectrumPassive2020}, follow-up works have extended analysis of Navier-Stokes Batchelor spectrum to settings where the fluid is degenerately forced \cite{cooperman2024exponential,nersesyan2024chaotic}. Other papers close to our work include that of Feng and Iyer \cite{feng2019dissipation}, who introduce pulsed-diffusion models with CAT maps as models for passive scalar advection in the free decay problem (absence of source term); and of Doering and Miles \cite{miles2018shell}, who study the free decay problem with a toy shell model indicating the important role played by the `leaky transmission' of Fourier mass from large to small scales. For additional works on small scale formation in pulsed diffusion models with spatially-regular advection, see, e.g., \parencite{elgindiOptimalEnhancedDissipation2023, cooperman2024harris, blumenthalExponentialMixingRandom2022}.




The spectral theory of transfer operators invoked in Section \ref{sec:sectorial} is part of a vibrant subfield of dynamical systems theory focused on statistical properties of various classes of chaotic dynamical systems exhibiting `shear-straining', which in dynamics literature is referred to as \emph{hyperbolicity} -- a reference to the saddle-like behavior of expansion/contraction occurring along hyperbolic trajectories. We refer to the texts \parencite{baladi2000positive,demersTransferOperatorsHyperbolic2021} for further discussion. Spectral stability in these frameworks, a core part of our approach in Section \ref{sec:sectorial}, is due to Keller, Liverani, Gouzel and others \parencite{gouezelBanachSpacesAdapted2005a,kellerStabilitySpectrumTransfer1999}. Finally, we point out that the specific function spaces used in Section \ref{sec:sectorial} are inspired by those introduced by Baladi and Tsujii \parencite{baladiAnisotropicHolderSobolev2007} (see also \cite{jezequel2023minicourse})  and Faure-Roy-Sj\"{o}strand \parencite{faure2008semi,faure2006ruelle}. 

\subsection*{Plan for the paper}

    Background on pulsed-diffusion and a general criterion for the cumulative version of Batchelor's law are given in Section \ref{sec:background}. We turn in Section \ref{sec:sectorial} to the proof of Theorem \ref{thm:sectorialIntro} in a somewhat more general setting, with some technical work deferred to Appendix \ref{sec:APP}. Finally, Section \ref{sec:dyadic} gives the precise statements and proofs of Theorems \ref{thm:pulseLocalizationIntro}, \ref{thm:expShellLawIntro}.

\section{Background}\label{sec:background}

\newcommand{\Lcfk}{\Lc_{f,\kappa}}

    Our aim in this section is to present a straightforward proof of the cumulative law \eqref{eq:cumLawIntro} (c.f. Proposition \ref{prop:cumLawPerturbedIntro}) in the pulsed-diffusion setting, largely following \cite{bedrossianBatchelorSpectrumPassive2020} but in a simplified form. 
    
    \bigskip 

    \noindent {\bf Some notation.} In this paper, we write $O(\cdot)$ for the standard big-O notation. Often, we will include subscripts, e.g., $O_n(\cdot)$, to refer to a quantity bounded from above by a multiplicative constant depending on $n$. When present, subscripts of the form ``$\neg n$'' are included to emphasize quantities the multiplicative constant \emph{does not} depend on. Lastly, for nonnegative reals $a, b$ we will write $a \lesssim b$ if $a = O(b)$, with analogous subscript conventions. 
    
    \bigskip

    \noindent {\bf Setup.} We take $f : \T^d \to \T^d$ to be any smooth diffeomorphism preserving the volume on $\T^d$. The composition $\Lc_{f, \kappa} = e^{\kappa \Delta} \circ \Lc_f$ is as in the previous section, as is the Markov chain $(g_n)$ given by 
    \[g_{n + 1} = \Lcfk g_n + \omega_{n + 1} b \, . \] 
    For us, the precise form of the source term $b$ is unimportant; it suffices to assume $b$ is $C^\infty$ and mean-zero, with $\| b \| = 1$. Similarly, the precise law of the random variables $(\omega_n)$ is unimportant, so long as $(\omega_n)$ is an IID sequence of square-integrable random variables of a probability space $(\Omega, \Fc, \P)$ for which $\E [\omega_i] = 0$ and $\Var(\omega_i) = 1$ (the latter assumption taken for simplicity). 
    
    The norm $\| \cdot \|$ refers to the $L^2$ norm throughout. We will write $\| \cdot\|_{H^s}, s \in \R$ for the family of \emph{homogeneous $H^s$ seminorms} defined by 
    \[\| h\|_{H^s}^2 = \sum_{k \in \Z^d \setminus \{ 0 \}} |k|^{2 s} | \hat h(k)|^2 \,, \]
    noting that $\| \cdot \|_{H^s}$ as above is a true norm on the space of mean-zero functions in $H^s$ (the only ones we consider in this paper). 
    
    \bigskip
    
    The main result of this section uses the following crucial hypothesis. 
    \begin{definition}\label{defn:kappaUnifH-1Decay2}
        We say that $f$ exhibits \emph{$\kappa$-uniform $H^{-1}$ decay at rate $\gamma > 0$} if for all $h : \T^d \to \R$ smooth and mean-zero, there exists $C_h > 0$ such that 
        \[\| \Lc_{f, \kappa}^n h \|_{H^{-1}} \leq C_h e^{- \gamma n}\]
        for all $\kappa \geq 0$ and $n \geq 1$. 
        Note, crucially, that the rate $\gamma$ and the constant $C_h$ are assumed independent of the diffusivity $\kappa$. 
    \end{definition}
    
We will prove the following. 

\begin{theorem}\label{thm:cumLaw2} Assume $f$ exhibits $\kappa$-uniform $H^{-1}$ decay at rate $\gamma > 0$. Then the following hold.
    \begin{itemize}
        \item[(i)] For all $\kappa > 0$ the Markov chain $(g_n)$ admits a unique stationary measure $\mu_{f, \kappa}$ supported in $L^2$ and for which the distribution is given by 
        \begin{align}\label{eq:gfkappaRep2} g_{f, \kappa} = \sum_{n = 0}^{\infty} \eta_n \Lc^n_{f, \kappa} b\end{align}
        where $(\eta_n)$ is an IID sequence with the same law as the $\omega_n$'s. 
        \item[(ii)] There exists a constant $C > 0$, independent of $\kappa$,  with the following property. Let $\kappa > 0$ be small and let $2 \leq N \leq \kappa^{-1/2}$. Then, 
        \begin{align}\label{eq:cumLaw2}\E \| \Pi_{\leq N} g_{f, \kappa}\|^2 \in \left[\frac{1}{2 \Lambda} \log N - C, \frac{2}{\gamma} \log N + C\right]  \end{align}
        and 
        \begin{align} \label{eq:controlDissipRange} \E \|\Pi_{\geq \frac{1}{\sqrt{\kappa}}} g_{f, \kappa}\|^2 \leq C \,. \end{align}
    \end{itemize}
\end{theorem}
Above, $\Lambda > 0$ is given by 
\begin{align} \label{eq:defnLambda2} \Lambda = \log \sup_{x \in \T^d} | (D f_x)^{-1} | \,. \end{align}
As one can show, $\kappa$-uniform $H^{-1}$ decay at rate $\gamma$ implies that  $\gamma \leq \Lambda$ (Lemma \ref{lem:compareGammaLambda2}). 

For details and for further discussion on the discrepancy between the upper and lower bounds in \eqref{eq:cumLaw2}, see Section \ref{subsec:discussion2}. 

Following some brief remarks, in Sections \ref{subsec:pfPartI2} -- \ref{subsec:upperBoundDissipative2} we prove Theorem \ref{thm:cumLaw2}. Some additional discussion on the mismatch between the upper and lower bounds in \eqref{eq:cumLaw2} is given in Section \ref{subsec:discussion2}. 

\begin{remark}[Use of $H^{-1}$-decay]
    The assumption of $\kappa$-uniform $H^{-1}$ decay is only used in the proof of the cumulative law upper bound. Indeed, all other aspects of the statement of Theorem \ref{thm:cumLaw2} -- all of part (i), the cumulative law lower bound, and the upper bound in the dissipative range -- are true for \emph{any $C^1$ diffeomorphism $f : \T^d \to \T^d$ which preserves volume.}
\end{remark}

\begin{remark}[Replacing $H^{-1}$ with $H^{-s}$ for $s > 0$]
    It changes nothing to replace the assumption of $\kappa$-uniform $H^{-1}$-decay to that of $H^{-s}$ decay for any $s > 0$. Indeed, it is straightforward to check that if $\kappa$-uniform $H^{-s}$ holds at some rate $\gamma_s > 0$, the cumulative law upper bound from \eqref{eq:cumLaw2} now becomes 
    \[\E \| \Pi_{\leq N} g_{f, \kappa}\|^2 \leq \frac{2s}{\gamma} \log N + C \,. \]
\end{remark}



\subsection{Proof of Theorem \ref{thm:cumLaw2}(i)}\label{subsec:pfPartI2}

Let $g_0, h_0 \in L^2$ be arbitrary fixed mean-zero scalars. Let $(g_n), (h_n)$ denote the scalar Markov chains initiated at $g_0$ and $h_0$, respectively, with the same random sample $\omega_n$; that is, on expanding \eqref{eq:defnPulseDiffIntro}, 
\begin{gather*}\label{eq:defnCoupling} g_n = \Lcfk^n g_0 + z_n \, , \qquad h_n = \Lcfk^n h_0 + z_n \, ,  \end{gather*}
where 
\[z_n := \omega_1 \Lcfk^{n-1} b + \dots + \omega_{n-1} \Lcfk b + \omega_n b \, , \]
treating both $(g_n)$ and $(h_n)$ as random sequences of $L^2$ random variables on the same probability space $(\Omega, \Fc, \P)$. 

The following will be sufficient to establish both existence and uniqueness of a stationary measure in $L^2$ as well as the representation \eqref{eq:gfkappaRep2}: 
\begin{claim}\label{cla:suffExistUnique2}\ 
    \begin{itemize}
        \item[(i)] $\| g_n - h_n\|_{L^2} \to 0$ as $n \to \infty$; and 
        \item[(ii)] The sequence $(z_n)$ of random $L^2$ functions converges in distribution\footnote{Recall that a sequence of probability measures $(\nu_n)$ on a metric space $X$ converges in distribution to a probability measure $\nu$ if for all $\varphi : X \to \R$ bounded and continuous, $\int \varphi d \nu_n \to \int \varphi d \nu$ as $n \to \infty$. } to 
        \[z_\infty = \sum_{n=0}^\infty \eta_n \Lc^n_{f, \kappa} b\]
        where $(\eta_n)$ is another IID sequence of standard Normal random variables. 
    \end{itemize}
\end{claim}
Both parts of Claim \ref{cla:suffExistUnique2} are immediate from the fact that $\| \Lc_f\| = 1$ and that $\| e^{\kappa \Delta}\| < 1$ for all $\kappa > 0$, hence $\| \Lcfk\| < 1$. Note that all operator norms are taken on the space of mean-zero scalars in $L^2$. 

The following argument for sufficiency of Claim \ref{cla:suffExistUnique2} is more-or-less standard, and is included for completeness. 

\begin{proof}[Proof of Theorem \ref{thm:cumLaw2}(i) assuming Claim \ref{cla:suffExistUnique2}]
Let $\mu_{f, \kappa}$ denote the distribution of $z_\infty$ on $L^2$. That $\mu_{f, \kappa}$ is stationary is immediate from the fact that if $\omega$ is a standard normal random variable independent of the $(\eta_n)$, then 
\[\Lcfk z_\infty + \omega b = \omega b + \sum_{n=1}^\infty \eta_{n -1} \Lcfk^n b\]
has the same distribution as $z_\infty$. 

For uniqueness, we will check that any two ergodic stationary measures coincide; if this is true, then by the ergodic decomposition theorem, it follows that only one stationary measure exists and is automatically ergodic.  Let $\hat \mu_1, \hat \mu_2$ be any two ergodic stationary measures for the scalar Markov chain and let $\psi : L^2 \to \R$ be any bounded continuous function. By the ergodic theorem, it holds that for $\hat \mu_1$-a.e. $g_0 \in L^2$ and $\hat \mu_2$-a.e. $h_0 \in L^2$, one has 
\[\frac1n \sum_{i=0}^{n-1} \psi(g_i) \to \int \psi d \hat \mu_1 \, , \quad \frac1n \sum_{i=0}^{n-1} \psi(h_i) \to \int \psi d \hat \mu_2 \, , \] 
where without loss we may take the sample paths $(g_n), (h_n)$ as in \eqref{eq:defnCoupling} with the same forcing. Since $\psi : L^2 \to \R$ is bounded and continuous, Claim \ref{cla:suffExistUnique2}(i) implies immediately that 
\[\frac1n \sum_{i=0}^{n-1} \psi(g_i) - \frac1n \sum_{i=0}^{n-1} \psi(h_i) \to 0 \, , \]
hence $\int \psi d \hat \mu_1 = \int \psi d \hat \mu_2$. As the measures $\hat \mu_1, \hat \mu_2$ agree on all bounded continuous $\psi$, it follows that $\hat \mu_1 = \hat \mu_2$. This completes the proof. 
\end{proof} 



\begin{remark}
    We caution that while $\mu_{f, \kappa}$ is a Borel measure supported in $L^2$ for each fixed $\kappa > 0$, it is not the case that $(\mu_{f, \kappa})_{\kappa > 0}$ is tight as a family of measures on $L^2$. For more on the ideal turbulent law converged-to by the family $(\mu_{f, \kappa})$ as $\kappa \to 0$, see \cite{bedrossianBatchelorSpectrumPassive2020}.  

\end{remark}

\subsection{Proof of cumulative law upper bound}

The following identity will be used throughout the remainder of the proof of Theorem \ref{thm:cumLaw2}. 

\begin{lemma}\label{lem:identityForProjection}
    In the notation of Theorem \ref{thm:cumLaw2}(i), it holds that 
    \[\E \| A g_{f, \kappa}\|^2 = \sum_{0}^{\infty} \| A \Lcfk^n b\|^2 \]
    for any bounded linear operator $A : L^2 \to L^2$. 
\end{lemma}
\begin{proof}[Proof of Lemma \ref{lem:identityForProjection}]
    By definition, 
    \begin{align}
        \| A g_{f, \kappa}\|^2 = \sum_{m = 0}^\infty \sum_{n = 0}^{\infty} \eta_m \eta_n \langle A \Lcfk^m b, A \Lcfk^n b \rangle \,. 
    \end{align}
    On taking $\E$ of both sides, terms with $m \neq n$ will vanish, since $\E(\eta_n \eta_m) = 0$.  
\end{proof}

Turning now to the upper bound in \eqref{eq:cumLaw2}, let $C_b, \gamma > 0$ be as in Definition \ref{defn:kappaUnifH-1Decay2}. Let $N \geq 2$ be arbitrary. Applying Lemma \ref{lem:identityForProjection} with $A = \Pi_{\leq N}$, 
\begin{align}
    \E \| \Pi_{\leq N} g_{f, \kappa} \|^2 & = \sum_{n=0}^{\infty} \| \Pi_{\leq N} \Lcfk^n b\|^2 \\ 
    & \leq \sum_{n=0}^{\infty}  \min\{1, N^2 \| \Lcfk^n b\|_{H^{-1}}^2\} 
\end{align}
where in the second line we used Bernstein's inequality, and have used that $\|\Lcfk\| \leq 1, \| b \| = 1$. Applying the $H^{-1}$ decay estimate gives 
\[\E \| \Pi_{\leq N} g_{f, \kappa}\|^2 \leq \sum_{n=0}^{\infty} \min\{1, N^2 \cdot C_b e^{- \gamma n}\}\, .  \]
The latter term in the $\min\{ \cdots\}$ is smaller for all 
\begin{align}\label{eq:defnNlate2}
n \geq n_L := \frac{2}{\gamma} \log N + \frac{1}{\gamma} \log C_b \, , 
\end{align}
from which it follows that 
\begin{align*} \E \| \Pi_{\leq N} g_{f, \kappa}\|^2 & \leq n_L + C_b N^2 \sum_{n=\lceil n_L \rceil}^{\infty} e^{- n \gamma} \\ 
& \leq n_L + \frac{C_b N^2 e^{- n_L \gamma}}{1 - e^{- \gamma}} = \frac{2}{\gamma} \log N + C
\end{align*}
where $C > 0$ is independent of $\kappa$ and $N$. This is precisely the sought-after upper bound in \eqref{eq:cumLaw2}.

\subsection{Proof of cumulative law lower bound}

The following estimate will be used. 

\begin{lemma}\label{lem:enhancedDissipLower2}
    Assume $\kappa \in (0,1]$. Let 
    \[\tau_\kappa = \min\{ n \geq 0 : \| \Lcfk^n b \|^2 < \frac12 \| b \|^2 \} \, . \]
    Then, 
    \[\tau_\kappa \geq \frac{1}{2\Lambda} |\log \kappa| - C \, , \]
    where $C > 0$ is a constant independent of $\kappa$ and where $\Lambda$ is as in \eqref{eq:defnLambda2}. 
\end{lemma}
\begin{proof}[Proof of Lemma \ref{lem:enhancedDissipLower2}]
    To start, for mean-zero $h \in L^2$ we have the standard energy estimate 
    \begin{align}
        \| \Lcfk h\|^2 & = \| h\|^2 - \int_0^\kappa \| e^{t \Delta} \Lc_f h\|_{H^1} \; dt \geq \| h\|^2 - \kappa \| \Lc_f h \|_{H^1}^2 \\ 
        & \geq \| h \|^2 - \kappa e^{2 \Lambda} \| h \|_{H^1}^2 \, .
    \end{align}
    Here, we estimated 
    \begin{align*}
        \| \Lc_f h \|_{H^1}^2 & = \int_{\T^d} | \nabla [h \circ f^{-1}]|^2 dx \\ 
        & = \int_{\T^d} |(D(f^{-1})_{x})^\top [\nabla h(f^{-1} x)]|^2dx \\ 
        & = \int_{\T^d} |(Df_{f^{-1} x})^{-\top} [\nabla h(f^{-1} x)]|^2 dx \\ 
        & = \int_{\T^d} |(Df_{x})^{-\top} [\nabla h(x)]|^2 dx \\ 
        & \leq \left(\sup_{x \in \T^d} | (Df^{-1})_x|\right)^2 \cdot \| h \|_{H^1}^2 = e^{2 \Lambda} \| h \|_{H^1}^2
    \end{align*}
    with $\Lambda$ as in \eqref{eq:defnLambda2}. Note that since $\| e^{\kappa \Delta} \|_{H^s} \leq 1$ for all $s$, it follows that $\| \Lcfk\|_{H^1} \leq e^{\Lambda}$ for all $\kappa$. 

    Bootstrapping, we obtain that 
    \begin{align*}
        \| \Lcfk^n b\|^2 - \| b\|^2 & \geq - \kappa \sum_{j=0}^{n-1} \| \Lcfk^{j} b\|_{H^1}^2 \\ 
        & \geq - \kappa \sum_{j=0}^{n-1} e^{2 j \Lambda} \| b \|_{H^1} \geq - C \kappa e^{2 n \Lambda} 
    \end{align*}
    where $C > 0$ is constant in $n$ and $\kappa$. 

    We now have 
    \[\frac12 \| b\|^2 \geq \| \Lcfk^{\tau_\kappa} b\|^2 \geq \| b\|^2 - C \kappa e^{2 \Lambda \tau_\kappa} \, ,  \]
    hence 
    \[\log \kappa + 2 \Lambda \tau_\kappa\]
    is bounded from below by a constant independent of $\kappa$. This completes the proof. 
\end{proof}

For the lower bound in \eqref{eq:cumLaw2}, let $2 \leq N \leq \kappa^{-1/2}$ be fixed and define 
\[n_E = \frac{1}{\Lambda} \log N \,. \]
Using Lemma \ref{lem:identityForProjection}, we estimate 
\begin{align}
    \E \| \Pi_{\leq N} g_{f, \kappa}\|^2 \geq \sum_{n \leq n_E} \| \Pi_{\leq N} \Lcfk^n b\|^2 = \sum_{n \leq n_E} \| \Lcfk^n b\|^2 - \sum_{n \leq n_E} \| \Pi_{> N} \Lcfk^n b\|^2 \,. 
\end{align}
The latter is estimated by Bernstein's inequality: 
\[\sum_{n \leq n_E} \| \Pi_{> N} \Lcfk^n b\|^2 \leq \sum_{n \leq n_E} N^{-2} \| \Lcfk^n b\|_{H^1}^2 \leq C N^{-2} e^{2 n_E \Lambda} = C\]
where $C > 0$ is a constant independent of $\kappa, N$, and where above we used the bound $\| \Lcfk\|_{H^1} \leq e^\Lambda$. 
For the former term, since $N \leq \kappa^{-1/2}$ we have by Lemma \ref{lem:enhancedDissipLower2} that $n_E \leq \tau_\kappa$, hence
\[\sum_{n < n_E } \| \Lcfk^n b\|^2 \geq \frac12 n_E - \frac12 \, , \]
having used that $\| b \| = 1 $. This implies the desired lower bound in \eqref{eq:cumLaw2}. 

\subsection{Control in the dissipative range}\label{subsec:upperBoundDissipative2}

\newcommand{\pik}{\Pi_{\geq \frac{1}{\sqrt{\kappa}}}}

We finally come to the upper bound in the dissipative range \eqref{eq:controlDissipRange}. We require the following, which is an analogue of the $L^2$ energy balance relation \cite[Proposition 1.2]{bedrossianBatchelorSpectrumPassive2020}. 
\begin{lemma}\label{lem:energyBalance2}
    For any $\kappa > 0$, it holds that
    \[\| b \|^2 = \E \int_0^\kappa \| e^{t \Delta} \Lc_f g_{f, \kappa}\|_{H^1}^2 dt \,. \]
    In particular, 
    \[  \kappa \E \| \Lcfk g_{f, \kappa}\|^2_{H^1} \leq \| b\|^2 \,. 
    \]
\end{lemma}
\begin{proof}[Proof of Lemma \ref{lem:energyBalance2}]
    The energy balance equality follows from the semigroup property of $(e^{\Delta t})$ and is essentially the first line in the proof of Lemma \ref{lem:enhancedDissipLower2}. The upper bound on $\kappa \E \| \Lcfk g_{f, \kappa}\|_{H^1}^2$ is then immediate from the fact that $ \| e^{t \Delta} \Lc_f g_{f, \kappa}\|_{H^1}^2 \geq \| e^{\kappa \Delta} \Lc_f g_{f, \kappa}\|_{H^1}^2  $ for all $t\in[0,\kappa].$
\end{proof}



To complete the proof of \eqref{eq:controlDissipRange}, a modification of the argument for Lemma \ref{lem:identityForProjection} implies 
\[\E \| \pik g_{f, \kappa}\|^2 = \E \| \pik \Lcfk g_{f, \kappa}\|^2 + \| \pik b \|^2 \,. \]
Using Bernstein and the energy upper bound from Lemma \ref{lem:energyBalance2}, we conclude that 
\[\E \| \pik g_{f, \kappa}\|^2 \leq \kappa \E \| \Lcfk g_{f, \kappa}\|_{H^1}^2 + \| b \|^2 \leq 2 \| b \|^2 \, . \]
As the RHS is independent of $\kappa$, the proof is complete. 


\subsection{Further discussion}\label{subsec:discussion2}




We close this section with a discussion on the mismatch between the upper and lower bounds in the cumulative law \eqref{eq:cumLaw2}.

To start, we address the general relationship between the values $\Lambda, \gamma$ appearing there. 

\begin{lemma}\label{lem:compareGammaLambda2}
    Assume that $\kappa$-uniform $H^{-1}$ decay holds at rate $\gamma > 0$. Then, at $\kappa = 0$ one has 
    \[\gamma \leq \liminf_n \frac1n \log \| \Lc^n_f h\|_{H^1} \]
    for all $h$ with mean zero, not identically equal to zero. 
    In particular, $\gamma \leq \log \sup_x | (Df_x)^{-1}| = \Lambda$. 
\end{lemma}
\begin{proof}
    By a standard interpolation inequality, 
    \[\| h \|^2 = \| \Lc_f h\|^2 \leq \| \Lc_f^n h\|_{H^1} \cdot \| \Lc_f^n h\|_{H^{-1}} \, . \]
    We recall that \(
\gamma\leq\liminf_{n\to\infty}\left(-\frac{1}{n}\log\left\Vert \mathcal{L}_{f,\kappa}^{n}h\right\Vert_{H^{-1}}\right)
\). The proof is complete on plugging in $\kappa$-uniform $H^{-1}$ decay, rearranging terms, and taking $n \to \infty$. 
\end{proof}

It is not surprising that there is some general inequality between $\gamma$ and $\Lambda$. Indeed, decay of the $H^{-1}$ norm implies the evacuation of \emph{all} $L^2$-mass of the scalar at low modes, since for mean-zero $h \in L^2$ one has 
\[\| \Pi_{\leq N} h \|^2 \leq N^{2} \| h \|^2_{H^{-1}} \,, \]
while, growth of the $H^1$ norm implies that \emph{some} of the $L^2$ mass of the scalar is migrating to higher modes, due to the weighting in the definition of $H^1$. 
On the other hand, \emph{upper bounds} on the $H^1$ norm \emph{prevent} the accumulation of mass in the high modes, since 
\[\| \Pi_{> N} h \|^2 \leq N^{-2} \| h\|^2_{H^1} \,. \]

This perspective allows us to interpret the values the upper and lower bounds in \eqref{eq:cumLaw2}. Following Theorem \ref{thm:cumLaw2}(i), we can view $g_{f, \kappa}$ as a superposition of the time-$n$ ``packets'' of scalar $\Lcfk^n b$. Each packet $b_n := \Lcfk^n b$ has $L^2$ mass distributed in some way over the possible modes $k \in \Z^d$. At each fixed wavelength $N$, the proof of the upper bound in \eqref{eq:cumLaw2} treats $\Pi_{\leq N} b_n$ as negligible for $n > n_L$, where $n_L$ is approximately $\frac{2}{\gamma} \log N$. Put another way, this says that the Fourier mass of the time-$n$ packet $b_n$ is concentrated \emph{above} wavelengths $\approx e^{n \gamma / 2}$. Similarly, the proof of the lower bound in \eqref{eq:cumLaw2} treats $\Pi_{> N} b_n$ as negligible for $n < n_E = \frac{1}{\Lambda} \log N$, i.e., the Fourier mass of $b_n$ is concentrated \emph{below} wavelengths $\approx e^{n \Lambda}$.


This discussion is relevant to the validity of a version of Batchelor's law summed over exponentially-spaced shells as in Theorem \ref{thm:expShellLawIntro}, i.e., 
\[\E \| \Pi_{[L^\ell, L^{\ell + 1}]} g_{f, \kappa} \|^2 \approx_L 1 \,. \]
If one could match these `inner' and `outer' propagation speeds, i.e., the timescales $n_E, n_L$ as above, such an exponentially-spaced law would follow. However, it seems unlikely to make such an argument work for general systems, since the `upper' and `lower' propagation speeds come from essentially distinct mechanisms. 

On the other hand, the above heuristic does not rule out the possibility of a `universal' exponentially-spaced Batchelor's law. It merely suggests that proving such a law would require controlling cancellations between a growing number of time-$n$ packets $b_n$. The authors feel that this is a compelling issue meriting future study.

\section{Angular sparsity  for perturbations of CAT maps}\label{sec:sectorial}



Over the last several decades there has been developed an abstract framework for the quantitative study of \emph{correlation decay} for suitable smooth mappings $f$, which roughly speaking refers to the tendency of long-time trajectories of chaotic dynamical systems to gradually `forget' the initial condition. This relates directly to $H^{-1}$ decay: if $f : \T^d \to \T^d$ is a diffeomorphism preserving volume, $X$ is a random initial condition in $\T^d$ distributed like Lebesgue measure, and $\varphi, \psi : \T^d \to \R$ are real-valued observables, then the correlation coefficient of the pair $\psi(f^n(X))$ and $\varphi(X)$  is proportional, up to an $n$-independent constant, to 
\[\int \psi(f^n x) \cdot \varphi(x) \, dx - \int \psi \int \varphi \,. \]
Up to a change of variables, the decay of this quantity is clearly tied to $H^{-1}$ decay of $\varphi \circ f^{-n}$ when $\psi \in H^1$ and in the special case when both $\psi, \varphi$ have mean zero. 

The plan for this section is as follows. In Section \ref{subsec:decayMechanism3} we recall elements of the spectral theory of correlation decay and record a spectral-type sufficient condition for $\kappa$-uniform $H^{-1}$ decay of a volume-preserving mapping $f : \T^d \to \T^d$. 

The key tool here is the use of spaces of ``mixed regularity'' in different tangent directions in state space corresponding to the directions of stretching and compression, the class of so-called \emph{anisotropic Banach spaces}. 
In Section \ref{subsec:catMap3} we present the Arnold CAT map and small perturbations thereof as a worked example (Theorem \ref{thm:specPic3}), showing how the form of the anisotropic space implies the stratification estimate in Theorem \ref{thm:sectorialIntro}. 
Theorem \ref{thm:specPic3} is proved in Section \ref{subsec:proofs3} and Appendix \ref{Jezequel}.

\subsection{Overview of spectral theory for transfer operators}\label{subsec:decayMechanism3}

\newcommand{\Cov}{\operatorname{Cov}}





\newcommand{\rhoess}{\rho_{\rm ess}}

     Let $f : \T^d \to \T^d$ be a $C^\infty$ diffeomorphism, which as earlier will be assumed to preserve volume throughout. 
     As in the pulsed-diffusion model, of interest is the \emph{transfer operator} $\Lc_f$, which acts for measurable $\varphi : \T^d \to \R$ as \[\Lc_f[\varphi](x) = \varphi \circ f^{-1}\, . \]
     
     \begin{assumption}\label{ass:quasicompact3}
    There exists a norm $\| \cdot \|_s$ on $C^\infty$ functions $\T^d \to \R$ with the following properties. Below, $\Bc_s$ denotes the completion of $C^\infty$ with respect to the norm $\| \cdot \|_s$. 
    \begin{itemize}
        \item[(1)] $\Lc_f$ extends to a bounded linear operator of $(\Bc_s, \| \cdot \|_s)$ into itself. 
        \item[(2)] With $\rhoess(\Lc_f ; \Bc_s)$ being the essential spectral radius of $\Lc_f$ viewed as an operator on $\Bc_s$, it holds that 
        \[\rhoess(\Lc_f; \Bc_s) < 1 \,. \]
    \end{itemize}
    \end{assumption}
    
    Here $\| \cdot \|_s$ is the strong norm in the theory of spectral perturbation as in Section \ref{subsubsec:abstractKellerLiverani}.

     The following is a general spectral criterion for $\kappa$-uniform $H^{-1}$ decay. 
     \begin{proposition}\label{prop:suffCondH-kMixing}
        Let Assumption \ref{ass:quasicompact3} hold for a norm $\| \cdot\|_s$ with the additional property that there exist $q \geq \frac12, C \geq 1$ for which
        \begin{align}\label{eq:compatibleSobolev3} C^{-1} \| \varphi \|_{H^{-q}} \leq \| \varphi\|_s \leq C \| \varphi\|_{H^q}\end{align}
        for all smooth $\varphi$. 
        Additionally, assume that
        \[\sigma(\Lc_f; \Bc_s) \setminus B_r(0) = \{ 1 \}\]
        for some $r \in (\rhoess(\Lc_f; \Bc_s), 1)$, 
        where $\sigma(\Lc_f ; \Bc_s)$ refers to the spectrum of $\Lc_f : \Bc_s \to \Bc_s$, and where 1 is a simple eigenvalue with eigenfunction ${\bf 1} : \T^d \to \R$ identically equal to 1. 
        
        Then, $f$ exhibits $\kappa$-uniform $H^{-1}$ decay as in Definition \ref{defn:kappaUnifH-1Decay2}
    \end{proposition}
     \begin{proof}
        The spectral picture given in the hypotheses implies that for any $r \in (\rhoess(\Lc_f; \Bc_s), 1)$, there exists $C_r > 0$, perhaps depending on $r$, such that \[\| \Lc_f^n \varphi \|_s \leq C_r r^n \| \varphi\|_s\] for all smooth $\varphi$ with mean zero and for all $n \geq 1$. Observe also that $\Lc_f$ extends to a bounded linear operator on each of $H^q, H^{-q}$, and that $\| \Lcfk\|_{H^q} \leq \| \Lc_f\|_{H^q}, \| \Lcfk\|_{H^{-q}} \leq \| \Lc_f\|_{H^{-q}}$. Finally, recall the general estimate 
        \begin{align}\label{eq:heatSemigroupEst}\| e^{\kappa \Delta} - I\|_{H^{s+1} \to H^{s}} \leq C \sqrt{\kappa} \quad \text{ for all } \quad s \in \R \, ,  \end{align}
        with $C > 0$ a constant, 
        which implies immediately that $\| e^{\kappa \Delta} - I\|_{H^{q} \to H^{-q}} \leq \sqrt{\kappa}$ on setting $s = -1/2$, using that $q \geq 1/2$. 
        
        We now develop an $H^{-q}$ decay estimate for $\Lcfk$. Let $b \in C^\infty$ be fixed and with mean zero. It holds for all $n \geq 1$ that
        \begin{align*}
            \Lcfk^n & = \Lc_f^n + \sum_{0}^{n-1} \Lcfk^{i} (\Lcfk - \Lc_f) \Lc_f^{n-i-1} \,, \qquad \text{hence} \\ 
            \| \Lcfk^n b \|_{H^{-q}} & \leq \| \Lc_f^n b\|_{H^{-q}} + \sum_0^{n-1} \| \Lcfk^{i} (\Lcfk - \Lc_f) \Lc_f^{n-i-1} b \|_{H^{-q}} \\ 
            & \leq \| \Lc_f^n b\|_s + \sum_0^{n-1} \| \Lc_f\|_{H^{-q}}^i \| e^{\kappa \Delta} - I \|_{H^q \to H^{-q}} \| \Lc_f\|_{H^q}^{n-i} \| b \|_{H^q} \\ 
            & \leq C_r r^n \| b \|_{H^q} + n \kappa C^n_q \| b \|_{H^q} \,. 
        \end{align*}
        Here, $C_q \geq 1$ depends only on $q$. 
        
        Fix $n_0$ such that $C_r r^{n_0} < \frac13$. It follows that for $\kappa < (3 n_0 C^{n_0}_q)^{-1} =: \kappa_0$ we have the estimate $\| \Lcfk^{n_0} b\|_{H^{-q}} \leq \frac23 \| b \|_{H^q}$. Since $\Lcfk$ is bounded on $H^q$, it follows that there exist $C \geq 1, \gamma_q > 0$ such that 
        \[\| \Lcfk^n b\|_{H^{-q}} \leq C e^{- \gamma_q n} \| b \|_{H^q}\]
        holds for all $n \geq 1$, uniformly over $\kappa \in [0,\kappa_0]$. 

        It remains to pass from an $H^{-q}$ estimate to an estimate in $H^{-1}$. This follows from a standard mollification argument -- see, e.g., \cite[Lemma 7.1]{bedrossianAlmostsureExponentialMixing2019}. 
    \end{proof}
     Assumption \ref{ass:quasicompact3} is sometimes referred to as \emph{quasicompactness}, and is a stringent condition on $\| \cdot \|_s$ requiring fine control of the dynamics of $f$, specifically the expansion-contraction mechanism (a.k.a. hyperbolicity) responsible for correlation decay for smooth dynamics $f$. Roughly speaking, quasicompactness requires that $\| \cdot \|_s$ assign positive regularity along unstable directions and negative regularity along stable directions. Due to mix of regularities, such norms are often referred to as \emph{anisotropic}. 
     There is a long literature constructing such spaces when realization of expansion-contraction is \emph{uniform} across state space, as it is for the Arnold CAT map and other so-called \emph{Anosov diffeomorphisms}. For further discussion, see, e.g., the books \cite{baladi2000positive, demersTransferOperatorsHyperbolic2021}. 

     Particularly useful for us in the study of the distribution of Fourier mass of passive scalars will be classes of \emph{Sobolev-like} anisotropic norms developed using tools of semiclassical analysis: see, e.g., \cite{baladiAnisotropicHolderSobolev2007, faure2006ruelle, faure2008semi} in the case of uniformly hyperbolic diffeomorphisms. Our approach borrows from several sources, including \cite[Section 3.3]{demersTransferOperatorsHyperbolic2021} and \cite{faure2008semi}. 
     


\subsection{Anisotropic spaces for CAT maps}\label{subsec:catMap3}

\begin{definition} \label{defn:f_map}
    Let $A$ be any $2 \times 2$ matrix with integer entries and determinant $\pm 1$, and let $f_A(x) = A x$ mod 1 denote the corresponding mapping $\T^2 \to \T^2$ as in Section \ref{sec:intro}. We say that $A$ gives rise to a \emph{hyperbolic CAT map} $f_A$ if no eigenvalues of $A$ lie on the unit circle. In particular, it is easy to show that $A$ must admit two distinct, real, irrational eigenvalues with absolute values $\lambda, \lambda^{-1}$, respectively for some $\lambda > 1$.  

\end{definition}

We now set about defining a norm $\| \cdot \|_s$ tailored to the dynamics of $f_A$ (and mappings $f : \T^2 \to \T^2$ which are close to $f_A$) so that the quasicompactness condition (Assumption \ref{ass:quasicompact3}) is satisfied. Our presentation here is an adaptation of known methods for using microlocal analysis to define anisotropic norms-- that used here is an adaptation of \cite[Section 3.3]{demersTransferOperatorsHyperbolic2021}. 

\bigskip

\noindent {\bf Notation.} Let $v^u_T, v^s_T$ denote unit-vector eigenvectors of $A^{T}$ corresponding to eigenvalues outside and inside the unit circle, respectively. Define the cone 
\begin{align} \label{defn:uCone3}\Cc_u = \{ a^u v^u_T + a^s v^s_T : |a^s| \leq |a^u|\} \subset \R^2 \end{align}
of vectors roughly parallel to $v^u_T$, and $\Cc_s = (\Cc_u)^c$.  


For $\varphi : \T^d \to \R$ smooth and $p > 0$ fixed, our desired norm $\| \cdot\|_s = \| \cdot \|_{\Hc^p}$ will take the form
\[\| \varphi\|_{\Hc^p}^2 = | \hat \varphi(0)|^2 + \sum_{k \in \Cc_u \setminus \{ 0 \}} |k|^{2 p} | \hat \varphi(k)|^2 + \sum_{k \notin \Cc_u \setminus \{ 0 \}} |k|^{-2 p} | \hat \varphi(k)|^2 \, .  \]

Before continuing, we emphasize the two following points. 
\begin{itemize}
    \item[(i)] The norm $\| \cdot\|_{\Hc^p}$ assigns positive regularity for wavenumbers in the cone $\Cc_u$ and negative regularity for those in $\Cc_s$. 
    \item[(ii)] Let $\Hc_0^p$ denote the set of $\varphi \in \Hc^p$ for which $\hat \varphi(0) = 0$, which is identical with the completion with respect to $\| \cdot\|_{\Hc^p}$ of the set of mean-zero $C^\infty$ functions. Then, 
     \[H^{-p} \subset \Hc^p_0 \subset H^p\]
     where $H^s, s \in \R$ are the homogeneous Sobolev spaces defined in the beginning of Section \ref{sec:background}. In particular, $\Hc^p_0$ is compactly embedded in $H^{-L}$ for all $L > p$. 
\end{itemize}

\bigskip

\noindent {\bf Spectral picture with respect to $\| \cdot \|_{s} = \| \cdot \|_{\Hc^p}$}

\medskip

We now turn to the spectral properties of the transfer operator $\Lc_{f}$ of a smooth diffeomorphism $f$ appropriately close to $f_A$. Below, we fix the following norm compatible with the $C^\infty$ topology on diffeomorphisms $f, g : \T^2 \to \T^2$: 

\[d_{C^\infty}(f, g) = \sum_{k = 0}^\infty 2^{-k} \frac{d_{C^k}(f, g) }{1+d_{C^k}(f, g) }\, , \]
where $d_{C^0}$ is the uniform norm, and $d_{C^k}(f, g) = \max\{ d_{C^{k-1}}(f, g), \sup_{x \in \T^2} \| D^k f(x) - D^k g(x)\|\}. $

\begin{theorem}\label{thm:specPic3}
    Let $\nu \in (1, \lambda), p > 0$ and $r > \nu^{-p}$. Then, there exists $\epsilon, K > 0$ such that the following hold for all  $C^\infty$ volume-preserving diffeomorphisms $f : \T^2 \to \T^2$ such that 
    \begin{align}\label{eq:C1Close3}
        d_{C^\infty}(f, f_A) < \epsilon \, .
    \end{align}
    \begin{itemize}
        \item[(a)] $\Lc_f$ satisfies Assumption \ref{ass:quasicompact3} with norm $\| \cdot \|_s$, i.e., $\Lc_f$ extends to a bounded linear operator on $\Bc_s$, and satisfies the estimate
        \[\rhoess(\Lc_f; H^{m}) \leq \nu^{-p} \,; \]
        
        \item[(b)] it holds that $\sigma(\Lc_f; \Hc^p) \setminus B_r(0)$ consists of the simple eigenvalue 1 with corresponding eigenfunction ${\bf 1}$ identically equal to 1;  
        \item[(c)] for any $\varphi \in \Hc^p_0$, it holds that 
        \[\| \Lc_{f}^n \varphi \|_{\Hc^p} \leq K r^{n} \| \varphi \|_{\Hc^p} \quad \text{ for all } \quad n \geq 0 \, ;   \]
        and 
        \item[(d)] There exists $\kappa_0 > 0$, depending only on $A, \nu, p$ and $r$, such that for any $\kappa \in [0,\kappa_0]$, items (a) -- (c) hold with $\Lc_f$ replaced by  $\Lc_{f, \kappa}$, with constants independent of $\kappa$. 
    \end{itemize}
\end{theorem}


The following is immediate from Theorem \ref{thm:specPic3}(a) -- (c) and Proposition \ref{prop:suffCondH-kMixing}. 
\begin{corollary}\label{cor:pertCATCumBatchelor3}
    Any mapping $f$ as in the statement of Theorem \ref{thm:specPic3} satisfies $\kappa$-uniform $H^{-1}$ mixing, and therefore satisfies the cumulative version of Batchelor's law in Theorem \ref{thm:cumLaw2}. 
\end{corollary}


\begin{remark} \ 
    \begin{enumerate}
        \item Theorem \ref{thm:specPic3}(d) is stronger than what is needed for Corollary \ref{cor:pertCATCumBatchelor3}, but will be used shortly in the proof of Theorem \ref{thm:sectorialIntro}. 
        \item The proof of Theorem \ref{thm:specPic3} given in Section \ref{subsec:proofs3} and  Appendix \ref{Jezequel} only requires control on $d_{C^k}(f, f_A)$ for some  finite order $k \geq 1$ depending on the parameter $p$. See Remark \ref{rmk:unifLYcomments3} for additional comments. 
    \end{enumerate}
\end{remark}

\subsubsection*{Stratification estimate and proof of Theorem \ref{thm:sectorialIntro}}

With Theorem \ref{thm:specPic3} in hand, we can now turn to the proof of Theorem \ref{thm:sectorialIntro}, restated below for convenience. 

\begin{theorem}
Let $p > 0$. Then, there exists $\epsilon, C, \kappa_0 > 0$ such that the following holds for all $\kappa \in (0,\kappa_0]$ and $C^\infty$ volume-preserving diffeomorphisms $f : \T^2 \to \T^2$ with $d_{C^\infty}(f, f_A) < \epsilon$. 

With $g_{f, \kappa}$ a statistically stationary scalar for the pulsed diffusion process as in Theorem \ref{thm:cumLaw2} (c.f. Corollary \ref{cor:pertCATCumBatchelor3}) and for any $k \in \Cc^u$, it holds that 
\begin{align}
    \E | \hat g_{f, \kappa}(k)|^2 < \frac{C}{|k|^{2p}}  \,. 
\end{align}
\end{theorem}


\begin{proof}[Proof assuming Theorem \ref{thm:specPic3}]
    
    Lemma \ref{lem:identityForProjection} implies   
    \begin{align}
        \E | \hat g_{f, \kappa}(k)|^2 = \sum_{n = 0}^\infty | \widehat{(\Lcfk^n b)} (k)|^2 \, . 
    \end{align}
    Fix $\nu \in (1, \lambda), r \in (\nu^{-p},1)$. Using Theorem \ref{thm:specPic3}, for $k \in \Cc^u$, it holds for all $n \geq 0$ that 
    \[|k|^{2p} | \widehat{(\Lcfk^n b)}(k)|^2 \leq \| \Lcfk^n b\|_{\Hc^p} \leq K r^n \| b \|_{\Hc^p} \, , \]
    hence 
    \[\E | \hat g_{f, \kappa}(k)|^2 \lesssim |k|^{-2p} \]
    up to a multiplicative constant depending only on $K, \nu, r, p$ and $\| b \|_{\Hc^p}$. 
\end{proof}

\subsection{Proofs}\label{subsec:proofs3}

It is natural to expect to obtain the spectral picture for $\Lc_f, \Lc_{f, \kappa}$ by perturbing off of the spectral picture for that of $\Lc_{f_A}$. Even in our setting, though, the perturbation $\Lc_{f_A} \mapsto \Lc_{f, \kappa}$ is very badly singular, and some controls will be needed to ensure spectral stability.


In Section \ref{subsubsec:abstractKellerLiverani} below, we present an abstract framework for spectral perturbation in this setting due to Keller and Liverani \cite{kellerStabilitySpectrumTransfer1999}. This framework is applied to the perturbation $f$ of $f_A$ in Section \ref{subsubsec:completeProof3}, where we summarize what will be needed. The proof of the main ingredient, Proposition \ref{prop:uniformLasotaYorke3} (Uniform Lasota-Yorke), is deferred to Appendix \ref{Jezequel}. 

\subsubsection{Spectral perturbation theory for transfer operators}\label{subsubsec:abstractKellerLiverani}

\newcommand{\vertiii}[1]{{\left\vert\kern-0.25ex\left\vert\kern-0.25ex\left\vert #1 
    \right\vert\kern-0.25ex\right\vert\kern-0.25ex\right\vert}}

Let $\| \cdot\|_s$ be a norm on $C^\infty$ functions on a compact manifold $M$, which we will refer to here as the \emph{strong norm};  as before we will write $\Bc_s$ for the completion of $C^\infty$ w.r.t. $\| \cdot \|_s$. In this framework we will require as well a second norm $\| \cdot\|_w$ on $C^\infty$, the \emph{weak norm}, with corresponding completed space $\Bc_w$. 

\begin{assumption}\label{ass:spacesAbstract3} \ 
    \begin{enumerate}
        \item For all $\varphi \in C^\infty$ it holds that 
        \[\|\varphi\|_w \leq \| \varphi\|_s \,.  \]
        \item Regarding $\Bc_s$ as a subspace of $\Bc_w$, it holds that 
        \begin{enumerate}
            \item $\Bc_s$ is dense in $(\Bc_w, \| \cdot \|_w)$; and 
            \item the unit ball $\{ \varphi \in \Bc_s : \| \varphi\|_s \leq 1\}$ is compact in $(\Bc_w, \| \cdot \|_w)$. 
        \end{enumerate}
    \end{enumerate}
\end{assumption}

Below, for an operator $Q : \Bc_s \to \Bc_s$ we define the `mixed norm'
\[\vertiii{Q} = \sup\{ \| Q v \|_w : v \in \Bc_s, \| v \|_s \leq 1 \} \,.\]

\medskip

Of interest for us are families $\mathcal{P}$ of operators on $\Bc_s$ nearby to some fixed $P_0 : \Bc_s \to \Bc_s$ for which $\sigma(P_0 : \Bc_s \to \Bc_s)$ is understood. 

\begin{assumption}\label{ass:operatorsAbstract3}
    There exist $C_1, C_2, C_3, M > 0$ and $\alpha \in (0,1)$ such that the following hold for all $P \in \mathcal{P}$. 
    \begin{enumerate}
        \item $P$ extends to a bounded linear operator $\Bc_w \to \Bc_w$ such that $\| P^n \|_w \leq C_1 M^n$ for all $n\in \mathbb{N}$; 
        \item (Uniform Lasota-Yorke) For all $\varphi \in \Bc_s$ and $n\in \mathbb{N}$, it holds that 
        \begin{align}\label{eq:unifLasotaYorke}
            \| P^n \varphi \|_s \leq C_2 \alpha^n \| \varphi\|_s + C_3 M^n \| \varphi \|_w \,. 
        \end{align}
    \end{enumerate}
\end{assumption}

\begin{theorem}[\cite{kellerStabilitySpectrumTransfer1999}] \label{thm:abstractStability3}
    Assume the setting of Assumptions \ref{ass:spacesAbstract3} and \ref{ass:operatorsAbstract3}. Then, for any $\eta > 0$ and $r \in (\alpha, \infty)$, there exist $\delta > 0, K > 0$ such that if $P \in \mathcal{P}, \vertiii{P - P_0} < \delta$, then: 
    \begin{itemize}
        \item[(a)] For all $z \in \sigma(P) \setminus B_r(0)$, there is some $z_0 \in \sigma(P_0)$ such that $|z - z_0| < \eta$; and 
        \item[(b)] it holds that 
        \[\| P^n \Pi^{(r)}_P\|_s \leq K r^n \qquad \text{ for all } n \geq 1 \, ,  \] so long as $\sigma(P_0: \Bc_s \to \Bc_s) \cap \{ |z| = r\} = \emptyset$. 
    \end{itemize}
\end{theorem}
Above, $\Pi_P^{(r)} = \frac{1}{2 \pi i} \int_{\gamma_r} (z - P)^{-1} dz$ is the spectral projector of $P$ corresponding to the part of $\sigma(P : \Bc_s \to \Bc_s)$ contained within the curve $\gamma_r$ of radius $r$ centered at zero. We note that as a consequence of part (a), $\sigma(P : \Bc_s \to \Bc_s) \cap \{ |z| = r\} = \emptyset$, and so $\Pi^{(r)}_P$ is defined for such $P$. 

\begin{remark}\label{rmk:singleOperator3}
    For a single operator $P$, Assumptions \ref{ass:spacesAbstract3} and \ref{ass:operatorsAbstract3} for some $\alpha \in (0,1)$ imply Assumption \ref{ass:quasicompact3}(2), namely, that 
    \begin{align}\label{eq:abstControlEssSpec}\rho_{\rm ess}(P : \Bc_s \to \Bc_s) \leq \alpha < 1 \,.  \end{align}
    This standard result, sometimes referred to as Hennion's theorem, is a consequence of Nussbaum's characterization of the essential spectral radius -- for details, see, e.g., \cite[Appendix B]{demersTransferOperatorsHyperbolic2021}. 
    Equation \eqref{eq:unifLasotaYorke} is often referred to as a \emph{Lasota-Yorke inequality}. 

    The hypotheses of Theorem \ref{thm:abstractStability3} ensure \eqref{eq:abstControlEssSpec} holds uniformly across $P \in \mathcal{P}$, hence one only controls the discrete spectrum (a well-behaved subset of the point spectrum) in Theorem \ref{thm:abstractStability3}(a). For further consequences of this framework, e.g., estimates on spectral projectors, we refer to the original paper \cite{kellerStabilitySpectrumTransfer1999} or to \cite[Appendix C]{demersTransferOperatorsHyperbolic2021}. 
\end{remark}

\subsubsection{Proof of Theorem \ref{thm:specPic3}}\label{subsubsec:completeProof3}

The following addresses a uniform Lasota-Yorke inequality for the set of transfer operators $\Lcfk$ for $f$ near to $f_A$ and for $\kappa \geq 0$. Its proof is given in Appendix \ref{Jezequel}. 

\begin{proposition}[Uniform Lasota-Yorke inequality]\label{prop:uniformLasotaYorke3}
    Let $p > 0, L > p$ and $\nu \in (1, \lambda)$. Then, there exists $\epsilon, C_2, C_3, M > 0$ such that the following holds. 
    Assume $\kappa \geq 0$ and that $f$ is a volume-preserving diffeomorphism of $\T^d$ such that $d_{C^\infty}(f, f_A) < \epsilon$. Then, for all $n \geq 1$ and $\varphi \in \Hc^p$, 
    \[\| \Lc_{f, \kappa}^n \varphi \|_{\Hc^p} \leq C_2 \nu^{-np} \| \varphi \|_{\Hc^p} + C_3 M^n \| \varphi\|_{H^{-L}} \,.  \]
\end{proposition}

Turning now to the proof of Theorem \ref{thm:specPic3}, we begin by recording the spectral picture for $\Lc_{f_A}$ on $\Hc^p$, before we perturb the transfer operator.

\begin{lemma}\label{lem:specPicCAT3}
    Let $p > 0$. Then, $\Lc_{f_A}$ extends to a bounded linear operator on $\Hc^p$, and 
    \[\sigma(\Lc_{f_A} : \Hc^p \to \Hc^p) \setminus \overline{B_{\lambda^{-p}}(0)} = \{1\} \, , \]
    where $1$ is a simple eigenvalue with eigenfunction ${\bf 1}$ identically equal to 1. 
\end{lemma}
\begin{proof}[Proof of Lemma \ref{lem:specPicCAT3}]
    For any $\nu \in (1, \lambda)$, Proposition \ref{prop:uniformLasotaYorke3} implies, in particular, that $\rhoess(\Lc_{f_A} : \Hc^p \to \Hc^p) \leq \nu^{-p}$ (c.f. Remark \ref{rmk:singleOperator3}), hence $\rhoess(\Lc_{f_A} : \Hc^p \to \Hc^p) \leq \lambda^{-p}$ on taking $\nu \to \lambda$. The spectrum complement of the essential spectrum is a subset of the point spectrum, so it remains to show that if $\Lc_{f_A} \varphi = c \varphi$ for some $|c| > \lambda^{-p}$ and $\varphi \in \Hc^p$, then $\varphi$ is constant almost-surely. 
    
    For this, consider the relation $\Lc_{f_A}^n \varphi = c^n \varphi$, $n \geq 1$, and take a Fourier transform of both sides, yielding
    \[c^n \hat\varphi(k) =\hat\varphi((A^T)^n k) \,. \]
    Fix $k_0 \in \Z^2 \setminus \{ 0 \}$. 
    Using that the eigenvectors $v^u_T, v^s_T$ of $A^T$ are not rational, it follows that there exists $N_0$ such that $(A^T)^n k_0 \in \Cc^u$ for all $n \geq N_0$. Now, 
    \begin{align*} \infty & > \| \varphi\|_{\Hc^p}^2 \geq \sum_{n \geq N_0} |(A^T)^n k_0|^{2p} |\varphi ((A^T)^n k_0)|^2 \\ 
    & \geq C^{-1} |\hat \varphi(k_0)|^2 \sum_{n\geq N_0} \lambda^{2pn} |c|^{2n} \, ,   \end{align*}
    where $C > 0$ is a constant. Since $|c| \lambda^p > 1$, the right-hand series diverges, and it follows that $\hat \varphi(k_0) = 0$. 
    
    We conclude that $\hat{\varphi}(0)$ is the only nonzero Fourier coefficient, hence $\varphi$ is constant and $c = 1$. From this argument, we also deduce that $1$ is a simple eigenvalue, as desired. 
\end{proof}

    \begin{proof}[Proof of Theorem \ref{thm:specPic3}]
        Fix $p, C_0 > 0, L \geq p + 2$ and $\nu \in (1, \lambda)$. Let $\epsilon > 0$ and $C_2, C_3, M > 0$ be as in Proposition \ref{prop:uniformLasotaYorke3}. 
        From the definitions, Assumption \ref{ass:spacesAbstract3} holds for the pair of spaces $\Bc_s = \Hc^p_0, \Bc_w = H^{-L}$. Assumption \ref{ass:operatorsAbstract3} for 
        \[\mathcal{P} = \{ \Lc_{f, \kappa}: \kappa \geq 0,  d_{C^\infty} (f_A, f) < \epsilon\}\] follows from Proposition \ref{prop:uniformLasotaYorke3}.

        Next, apply Theorem \ref{thm:abstractStability3} to the family $\mathcal{P}$ and $P_0 = \Lc_{f_A}$ with $\eta = \frac{1}{1000}$, yielding $\delta > 0$ such that $\vertiii{\Lc_{f_A} - \Lc_{f, \kappa}} < \delta$ implies items (a), (b) of Theorem \ref{thm:abstractStability3}. 
        Let us now check that if $\epsilon$ and $\kappa_0 > 0$ are sufficiently small, then $\vertiii{\Lc_{f_A} - \Lc_{f, \kappa}} < \delta$ whenever $d_{C^\infty}(f, f_A) < \epsilon$. Indeed, if $f_1, f_2 : \T^d \to \T^d$ then one has the general estimate 
        \[\| \Lc_{f_1} \psi - \Lc_{f_2} \psi\|_{H^s} \leq C_s d_{C^{\lceil |s| \rceil}} (f_1^{-1}, f_2^{-1}) \| \psi\|_{H^{s + 1}}\]
        for all $\psi \in C^\infty$ and $s \in \R$, where $C_s$ is a constant. In particular, $d_{C^{\lceil |s| \rceil}} (f_1^{-1}, f_2^{-1})$ is arbitrarily small for $f_1,f_2$ near $f_A$ in the $C^\infty$ metric. On the other hand, the difference between $\Lc_f$ to $\Lc_{f, \kappa}$ is addressed by the estimate \eqref{eq:heatSemigroupEst}. The desired estimate for $\vertiii{\Lc_{f_A} - \Lc_{f, \kappa}}$ follows on setting $s = -L$, using that $\| \cdot \|_{H^{-(L-1)}} \leq \| \cdot\|_{\Hc^p}$ by construction. 
        
        By now, we have deduced $\rhoess(\Lc_{f, \kappa}) \leq \nu^{-p}$ (Remark \ref{rmk:singleOperator3}). From Theorem \ref{thm:abstractStability3}(a) and the spectral picture for $\Lc_{f_A}$ in Lemma \ref{lem:specPicCAT3},  it follows that there is no point spectrum of $\Lc_{f, \kappa} : \Hc^p_0 \to \Hc^p_0$ away from $B_r(0)$. In view of volume preservation of $f$ and basic properties of the heat semigroup, it follows that $\sigma(\Lc_{f, \kappa} :\Hc^p \to \Hc^p) \setminus B_r(0) = \{ 1\}$,  where 1 is a simple eigenvalue. 
        Finally, the estimate in Theorem \ref{thm:specPic3}(c) follows from Theorem \ref{thm:abstractStability3}(b). This completes the proof.

        

    \end{proof}

\begin{remark}\label{rmk:unifLYcomments3}
    Before moving on, we comment on the various smoothness assumptions required in Section \ref{sec:sectorial}. 
    \begin{enumerate}
        \item The proof of Proposition \ref{prop:uniformLasotaYorke3} actually only uses that $f, f_A$ are close in $C^1$, and that $f$ itself has controlled $k$-th order derivatives for some $k$ depending on $p$. See Appendix \ref{Jezequel} for further details. 
        \item The proof of Theorem \ref{thm:specPic3} required control of $\Lc_{f, \kappa} - \Lc_{f_A}$ in the mixed operator norm $\vertiii{\cdot}$, which in turn required control on $d_{C^k}(f, f_A)$ for some $k$ depending on $L, p$.  So, the conclusions of Theorem \ref{thm:sectorialIntro} apply to $C^k$-small perturbations of the original CAT map $f_A$. 
    \end{enumerate}
\end{remark}

\section{Pulse localization for perturbations of CAT maps}\label{sec:dyadic}

\newcommand{\Ncrit}{N_{\rm crit}}

We now turn to Theorem \ref{thm:pulseLocalizationIntro}, concerning the localization of scalar mass, and Theorem \ref{thm:expShellLawIntro} on a version of Batchelor's law summing over radial shells of exponentially-growing width. Statements are given in Section \ref{subsec:statements4} immediately below. Our core tool, that of \emph{spectral distributions}, is presented next in Section \ref{subsec:spectralDistro4}. Proofs are given in Sections \ref{subsec:proofs4-1} and \ref{subsec:mainPropProof4}. 

\subsection{Statement of results}\label{subsec:statements4}

Let $A$ be as in the beginning of Section \ref{subsec:catMap3}, and let $f$ be a $C^\infty$ volume-preserving diffeomorphism; we will always assume $d_{C^\infty}(f, f_A) \leq \epsilon$, where $\epsilon \in (0,1]$ is small, and where $f_A : \T^2 \to \T^2$ is the hyperbolic CAT map  corresponding to $A$. 

Throughout, we assume that the driving term $b$ is of the form 
\[b = e_{k_0}\]
where $k_0$ is a fixed nonzero wavenumber. With $b$ fixed and $\kappa > 0$ small, the map $f$ satisfies $\kappa$-uniform $H^{-1}$-decay (Definition \ref{defn:kappaUnifH-1Decay2} and Corollary \ref{cor:pertCATCumBatchelor3}); let $ g_{f, \kappa}$ denote the corresponding statistically stationary scalar as in Theorem \ref{thm:cumLaw2}(a).  
With $k_0$ and $A$ fixed as above, let 
\[k_\ell = (A^{-T})^\ell k_0 \, , \qquad \ell \geq 1 \,. \]



\begin{theorem}\label{thm:pulseLocalization4}
  There exists $\bar \delta > 0$, depending only on the unperturbed map $f_A$ and the initial mode $k_0$, with the following properties. Let $\delta \in (0,\bar \delta)$; then, there exists $\zeta > 0$ such that for all $\kappa \in (0,1]$ and $\epsilon > 0$ sufficiently small, 
  it holds that 
  \begin{gather}\label{eq:pulseControl4}
  \E \sum_{\substack{k \notin \{ k_\ell\} \\ |k| \leq R_{\rm crit}}} |\hat g_{f, \kappa}(k)|^2  \lesssim \left( \max\{ \kappa, \epsilon\} \right)^{\delta} \, , 
  \end{gather}
  where
  \[R_{\rm crit} := \big(\max \{ \kappa, \epsilon\} \big)^{-\zeta} \,. \]
\end{theorem}


Next, we record a version of Batchelor's law summing over exponentially `thick' shells of wavenumbers. 

\begin{theorem}\label{thm:expShellLaw4}
  There exist $C, C' \geq 1$ such that the following holds for all $L > 1$ and $\epsilon, \kappa$ small. Define 
  \[\ell_{\rm crit} = \frac{C}{\log L} \left( | \log \max\{ \epsilon, \kappa\}| - C' \right) \,. \]
  Then, for all $\ell \leq \ell_{\rm crit}$, 
  \begin{align}\label{eq:exponentialShellEst4}
     \E \| \Pi_{[L^\ell, L^{\ell + 1}]} g_{f, \kappa}\|^2 = \frac{\log L}{\log \lambda} + O(1) \, , 
  \end{align}
  where $O(1)$ is an error term independent of $L$ and $\ell$. 
\end{theorem}


\subsection{Background: a probabilistic framework for Fourier mass} \label{subsec:spectralDistro4}


\begin{definition} 
Let $\varphi \in L^2$. We define the \emph{spectral distribution} $\PP_\varphi$ of $\varphi$ to be the probability measure on $\Z^2$ for which \[\PP_\varphi(\{ k \}) = \frac{1}{\| \varphi\|^2} |\hat \varphi(k)|^2 \,. \] 
\end{definition}
Below and what follows, we will write $X_\varphi$ for a $\Z^2$-valued random variable of some probability space $(\Omega, \Fc, \P)$ for which \[\P(X_\varphi = k) = \PP_\varphi(\{ k \}) \, ,\] i.e., $\PP_\varphi$ is the empirical law of $X_\varphi$. 

\begin{definition} \ 
  \begin{itemize}
    \item[(a)] The \emph{spectral centroid} $\E(X_\varphi)$ of $\varphi$ is defined to be the mean of the distribution $\PP_\varphi$, i.e., 
    \[\E X_\varphi = \sum_{k \in \Z^2} k \PP_\varphi(\{ k \}) = \frac{1}{\| \varphi\|^2} \sum_{k \in \Z^2} k |\hat \varphi(k)|^2 \,. \]
    \item[(b)] The \emph{spectral variance} of $\varphi$ is the variance $\Var(X_\varphi)$ of the distribution $\PP_\varphi$, i.e., 
    \[\Var(X_\varphi) = \sum_{k \in \Z^2} |k - \E X_\varphi|^2 \PP_\varphi(\{ k \}) \, . \]
  \end{itemize}
\end{definition}
Note that, like the usual notion of variance, the spectral variance satisfies the following alternative characterization: 
\[\Var(X_\varphi) = \E [|X_\varphi|^2] - | \E X_\varphi|^2 \,. \]

These definitions are inspired by the notions of power spectral density and spectral centroid in signal processing \parencite{oppenheimSignalsSystemsInference2016}.

We record below some basic consequences of these definitions, each a rendering of a classical probability result in the language defined above. 

\begin{lemma}\label{lem:basicSpectralProps4}\ 
  \begin{itemize}
    \item[(i)] For all $z \in \R^2$, we have that \[\Var(X_\varphi) \leq \E|X_\varphi - z|^2 \, ,\]
    with equality if and only if $z = \E X_\varphi$. 
    \item[(ii)] (Chebyshev's Inequality) For any $a > 0$, it holds that 
    \[\P(|X_\varphi - \E X_\varphi| > a ) \leq \frac{\Var(X_\varphi)}{a^2} \,. \]
  \end{itemize}
\end{lemma}


\subsection{Proof of Theorems \ref{thm:pulseLocalization4} and \ref{thm:expShellLaw4}} \label{subsec:proofs4-1}

We now turn to the use of the circle of ideas above towards controlling the distribution of Fourier mass of passive scalars. 
The main ingredient is the following estimate, which controls the centroid $\E X_{\varphi_n}$ and variance $\Var(X_{\varphi_n})$ of the iterated scalars \[\varphi_n := \Lcfk^n b, \varphi_0 := b = e_{k_0} \,. \] Its proof is deferred to Section \ref{subsec:mainPropProof4}.



\begin{proposition}\label{prop:mainVarCentEstimates4}
  There exists a constant $M > 1$, depending only on the unperturbed CAT map $f_A$, with the following property. Let $\epsilon, \kappa > 0$ be sufficiently small. For $n \lesssim_M |\log \kappa| + 1$, it holds that 
  \begin{gather*}
    \E X_{\varphi_n} = (A^{-T})^n k_0 + O(\max\{\kappa, \epsilon\} M^n) \, , \quad \text{ and}\\ 
    \Var(X_{\varphi_n}) \lesssim \max\{ \kappa, \epsilon\} M^n \,. 
  \end{gather*}
\end{proposition}

\begin{proof}[Proof of Theorem \ref{thm:pulseLocalization4}]
  Below, $\delta, \zeta > 0$ are fixed, with constraints on them to be specified as we go along. We write $\eta = \max\{ \kappa, \epsilon\}$, so that ultimately $R_{\rm crit} = |\log \eta|^\zeta$. Let $\Pi^\perp$ denote orthogonal projection onto the span of modes $k \notin \{ k_\ell\}$, and $\Pi_{\leq R_{\rm crit}}$ the projection onto the span of modes $k$ with $|k| \leq R_{\rm crit}$. Next, take $\epsilon$ small enough so that, by Theorem \ref{thm:specPic3}, the mapping $f$ admits $\kappa$-uniform $H^{-1}$-decay at some rate $\gamma$ (Definition \ref{defn:kappaUnifH-1Decay2}) uniformly over all $f$ with $d_{C^\infty}(f, f_A) < \epsilon$.

  By Lemma \ref{lem:identityForProjection}, the LHS of \eqref{eq:pulseControl4} can be represented as
  \[\E \|\Pi^\perp \Pi_{\leq R_{\rm crit}} g_{f, \kappa}\|^2 = \sum_{n = 0}^\infty \| \Pi^\perp \Pi_{\leq R_{\rm crit}} \varphi_n \|^2 \, . \]
  We split the RHS sum into the pieces $\{ n < N_{\rm crit}\}, \{ n \geq N_{\rm crit}\}$, where the \emph{critical time} $N_{\rm crit}$ will be specified shortly. 
  
  For the $n$-th summand, $n \leq N_{\rm crit}$, we have the upper bound
  \begin{align*}
    \sum_{k \neq k_n} | \hat \varphi_n(k)|^2 & = \| \varphi_n\|^2 \P(X_{\varphi_n} \neq k_n) = \| \varphi_n\|^2 \P\left(|X_{\varphi_n} - k_n| > \frac12\right) \\ 
    & \leq \P\left(| X_{\varphi_n} - \E X_{\varphi}| > \frac12 - |k_n - \E X_{\varphi_n}| \right) \,. 
  \end{align*}
  By Chebyshev's inequality (Lemma \ref{lem:basicSpectralProps4}(b)) and assuming $\frac12 - |k_n - \E X_{\varphi_n}| > 0$, we obtain the upper bound
  \begin{align*}
    \left( \frac12 - |k_n - \E X_{\varphi_n}| \right)^{-2} \Var(X_{\varphi_n}) \,, 
  \end{align*}
  which itself is controlled 
  \begin{align*}
    \lesssim_{\neg n} \eta M^n \quad \text{ if } \quad |k_n - \E X_{\varphi_n}| < \frac{1}{10} \,. 
  \end{align*}
  By Proposition \ref{prop:mainVarCentEstimates4}, the latter condition is guaranteed if $\eta M^n \ll 1$, and so we conclude that $\eta M^{\Ncrit} \ll 1$ implies  
  \[\sum_{n = 0}^{\Ncrit} \| \Pi^\perp \Pi_{\leq R_{\rm crit}} \varphi_n \|^2 \lesssim \eta M^{\Ncrit} \,. \]
Fixing 
\[\Ncrit = \frac{(1 - \delta) |\log \eta|}{\log M} \, , \]
  guarantees 
\[\sum_{n = 0}^{\Ncrit} \| \Pi^\perp \Pi_{\leq R_{\rm crit}} \varphi_n \|^2  \lesssim \eta M^{\Ncrit} \lesssim \eta^{\delta} \,. \]

To control $n > N_{\rm crit}$, by Bernstein's inequality and $\kappa$-uniform $H^{-1}$-decay we bound
\begin{align}\label{eq:controlHighModes4pulse}
  \| \Pi_{\leq R_{\rm crit}} \varphi_n\|^2 \leq R_{\rm crit}^2 \| \varphi_n\|^2_{H^{-1}} \lesssim R_{\rm crit}^2 e^{- 2 n \gamma} \, , 
\end{align}
hence 
\[\sum_{n > \Ncrit} \| \Pi^\perp \Pi_{\leq R_{\rm crit}} \varphi_n \|^2 \lesssim R_{\rm crit}^2 e^{- 2 \Ncrit \gamma} \,. \]
Expanding and plugging in $R_{\rm crit} = \eta^{-\zeta}$, we have 
\[R_{\rm crit}e^{ \Ncrit \gamma} = \eta^{-\zeta + \frac{\gamma(1 - \delta)}{\log M} } \,. \]
We shall set
\[\bar \delta := \frac{\gamma}{\log M + \gamma} \, , \qquad \zeta = \frac{\gamma}{2\log M} - \frac{\delta}{2} \left( 1 + \frac{\gamma}{\log M} \right)  \, , \]
so that the constraint $\delta \in (0,\bar \delta)$ ensures $\zeta > 0$ and $-\zeta + \frac{\gamma(1 - \delta)}{\log M} \geq \delta$. With this, the RHS of \eqref{eq:controlHighModes4pulse} is bounded $\lesssim \eta^\delta$, completing the proof. 
\end{proof}

\begin{proof}[Proof of Theorem \ref{thm:expShellLaw4}]
  As before, $\eta := \max\{ \kappa, \epsilon\}$. Let $\ell_{\rm crit} \geq 1$, to be specified as we go along. Below, $0 \leq \ell < \ell_{\rm crit}$ will be fixed; for short, write $\Pi_\ell = \Pi_{[L^\ell, L^{\ell + 1}]}$. 
  We also introduce the following notation: 
  \begin{gather*}
    \mathcal{K}_\ell^\circ = \{ n \geq 0 : L^\ell+1 \leq |k_n| \leq L^{\ell + 1} - 1\} \, , \\ 
    \overline{\mathcal{K}_\ell} = \{ n \geq 0 : L^\ell - 1 \leq |k_n| \leq L^{\ell + 1} + 1 \} \, , 
  \end{gather*}
  noting that the cardinalities $|\mathcal{K}_\ell^\circ|, |\overline{\mathcal{K}_\ell}|$ are each 
  \begin{align*}
    =  {\log L \over \log \lambda} + O_{\neg L}(1) \, , 
  \end{align*}
  where $O_{\neg L}(1)$ depends only on $A$ and $k_0$.


  Using Lemma \ref{lem:identityForProjection}, 
  \begin{align*}
    \E \| \Pi_{\ell} g_{f, \kappa}\|^2 = \sum_{n = 0}^\infty \| \Pi_\ell \varphi_n\|^2 \,. 
  \end{align*}
  In parallel to the proof of Theorem \ref{thm:pulseLocalization4}, we split the latter sum into $\{ n \leq \Ncrit\}, \{ n > \Ncrit\}$, where $\Ncrit$ is to be determined as we go. 

  We begin with the bound of $\E \| \Pi_{\ell} g_{f, \kappa}\|^2$ from above.  Starting with $\{ n \leq \Ncrit\}$, and using that $\| \varphi_n\| \leq 1$ for all $n$, 
  \begin{align*}
    \sum_{n = 0}^{\Ncrit} \| \Pi_\ell \varphi_n\|^2 \leq |\overline{\mathcal{K}_\ell}| + \sum_{\substack{n \notin \overline{\mathcal{K}_\ell} \\ n \leq \Ncrit}} \| \Pi_\ell \varphi_n\|^2  
  \end{align*}
  The term $|\overline{\mathcal{K}_\ell}|$ results in the primary contribution to 
  \eqref{eq:exponentialShellEst4}. For each summand $n \notin \overline{\mathcal{K}_\ell}$, we estimate
  \begin{align*}
    \| \Pi_\ell \varphi_n\|^2 & = \| \varphi_n\|^2 \P(|X_{\varphi_n}| \in [L^\ell, L^{\ell + 1}])  \\ 
    & \leq \P \left( |X_{\varphi_n} - k_n| \geq 1\right) \\ 
    & \leq \P \left( |X_{\varphi_n} - \E X_{\varphi_n}| \geq 1 - |k_n - \E X_{\varphi_n}| \right) \\ 
    & \leq {\Var(X_{\varphi_n}) \over (1 - |k_n - \E X_{\varphi_n}| )^2} \quad \text{if } |\E X_{\varphi_n} - k_n| < 1 \, . 
  \end{align*}
  By Proposition \ref{prop:mainVarCentEstimates4}, we are assured $|\E X_{\varphi_n} - k_n| \ll 1$ so long as $\eta M^{\Ncrit} \ll 1$, for which it suffices to choose $\Ncrit \leq | \log \eta | / \log M -K$, where $K > 0$ is sufficiently large (and independent of $L$). Imposing this constraint on $\Ncrit$ and summing over $n \leq \Ncrit$ gives
  \begin{align*}
    \sum_{\substack{n \notin \overline{\mathcal{K}_\ell} \\ n \leq \Ncrit}} \| \Pi_\ell \varphi_n\|^2  \lesssim \eta M^{\Ncrit} \ll 1 \,, 
  \end{align*}
  noting the above bound is independent of $L$. 
  
  The sum over $\{ n > \Ncrit\}$ is bounded via Bernstein's inequality: 
  \begin{align*}
    \sum_{n > \Ncrit} \| \Pi_\ell \varphi_n\|^2 & \leq \sum_{n > \Ncrit} L^{2\ell} \| \varphi_n\|^2 \\ 
    & \lesssim L^{2 \ell} e^{-2 \Ncrit \gamma} \, , 
  \end{align*}
  where $\gamma$ is as in the proof of Theorem \ref{thm:expShellLaw4}. 
  To ensure the above is $O(1)$ uniformly in $L$, it will be enough to take \[\ell_{\rm crit} = \frac{\Ncrit \gamma}{\log L} \,. \]
  Note that, while a further adjustment to $\Ncrit$ will be made below, the relationship between $\ell_{\rm crit}$ and $\Ncrit$ is now fixed. 
  
  For the lower bound, 
  \begin{align*}
    \sum_{n = 0}^\infty \| \Pi_\ell \varphi_n\|^2 \geq \sum_{n \in \mathcal{K}_\ell^\circ} \| \Pi_\ell \varphi_n\|^2 \,. 
  \end{align*}
  Writing $\Pi_\ell^\perp = I - \Pi_\ell$, note that $n \in \mathcal{K}_\ell^\circ$ implies 
  \begin{align*}
    \| \Pi_\ell^\perp\varphi_n\|^2 &= \| \varphi_n\|^2 \P(X_{\varphi_n} \notin[L^\ell, L^{\ell + 1}]) \\ 
    & \leq \P(|X_{\varphi_n} - k_n| \geq 1)  \\ 
    & \leq \P(|X_{\varphi_n} - \E X_{\varphi_n}| \geq 1 - |k_n - \E X_{\varphi_n}| ) \\ 
    & \leq {\Var(X_{\varphi_n}) \over (1  - |k_n - \E X_{\varphi_n}|)^2}
  \end{align*}
  provided, in the last line, that $|k_n - \E X_{\varphi_n}| < 1$. We point out, though, that $\ell + 1 \leq \ell_{\rm crit}$; since $n \in \mathcal{K}_\ell^\circ$ and since $|k_n| \sim \lambda^n$, it follows that 
  \[n \leq \ell_{\rm crit} \frac{\log L}{\log \lambda} + O(1) = \frac{\Ncrit \gamma}{\log \lambda} + O(1) \,. \]
  Perhaps on shrinking $\Ncrit$ further by an $O(1)$ term, uniformly in $L$, we can force $|k_n - \E X_{\varphi_n}| \leq 1/10$, say, so that in all we have $\| \Pi_\ell^\perp \varphi_n\|^2 \lesssim \Var(X_{\varphi_n}) \lesssim \eta M^{n}$. This implies that for $n \in \mathcal{K}_\ell^\circ$, 
  \begin{align*}
    \| \Pi_\ell \varphi_n\|^2 \geq \| \varphi_n\|^2 - O(\eta M^n) \, . 
  \end{align*}
  Borrowing from the proof of Lemma \ref{lem:estErrorVarianceCentroid4} (which is used to prove Proposition \ref{prop:mainVarCentEstimates4}),  $\| \varphi_n\|^2$ is bounded from below like 
  \begin{align*}
    \| \varphi_n\|^2 &\geq 1 - O(\kappa \sum_{\ell=0}^{n-1} \| \varphi_\ell\|_{H^1}^2)  \geq 1 + O(\eta M^n)
  \end{align*}
  where $M$ is as in Proposition \ref{prop:mainVarCentEstimates4}. 
  Summing up over $n \in \mathcal{K}^\circ_\ell$ now gives
  \begin{align*}
    \sum_{n \in \mathcal{K}_\ell^\circ} \| \Pi_\ell \varphi_n\|^2 & \geq |\mathcal{K}^\circ_\ell| + O(\eta M^{\Ncrit}) = \frac{\log L}{\log \lambda} + O(1) 
  \end{align*}
  as desired. 
\end{proof}

\subsection{Proof of Proposition \ref{prop:mainVarCentEstimates4}}\label{subsec:mainPropProof4}

\subsubsection*{Setup and preliminary estimates}

Below, $f$ is a fixed smooth, volume-preserving diffeomorphism of $\T^2$ with $d_{C^\infty}(f, f_A) < \epsilon$, where $\epsilon$ is a small parameter to be adjusted a finite number of times in what follows. At times we will abuse notation slightly and write 
\[\| f \|_{C^m} = \max_{1 \leq k \leq m} \sup_{x \in \T^2} | D^k f(x)| \,. \]

In what follows, we will make use of the Sobolev spaces $W^{k, p}$ defined for $\varphi \in C^\infty$ from the norm 
\[\| \varphi\|_{W^{k,p}} = \| \varphi \|_{W^{k-1,p}} + \left( \int_{\T^2} |D^k \varphi(x)|^p\right)^{1/p} \, , \]
recalling that for mean-zero functions, $\| \cdot \|_{W^{k,2}}$ is equivalent to the homogeneous $H^k$ norm $\| \cdot \|_{H^k}$ by Poincar\'e's inequality. 

We now record the following preliminary estimates. Below, 
\[\Rc_f := \Lc_f - \Lc_{f_A}\]
is the difference in transfer operators between $f$ and the CAT map $f_A$. 

\begin{lemma}\label{lem:sobolevEst4} For any $\varphi \in C^\infty$ with mean zero and for any $\kappa \geq 0$, the following hold. 
  \begin{itemize}
    \item[(a)] \begin{align*} 
      \| \Lc_{f,\kappa} \varphi\|_{W^{m, 1}} & \lesssim_m \| f \|_{C^m} \| \varphi\|_{W^{m,1}}  \\ 
    \| \Lc_{f,\kappa} \varphi\|_{H^m} & \lesssim_m \| f \|_{C^m} \| \varphi\|_{H^m}
  \end{align*}
  \item[(b)] For any $k \in \Z^2 \setminus \{ 0 \}$, \begin{align}\label{eq:sobolevEstRemainder4}
    |k|^m |\widehat{\Rc_f \varphi}(k)| \lesssim_m d_{C^m} (f^{-1}, f_A^{-1}) \| \varphi \|_{W^{m+1, 1}} \,. 
  \end{align}
  \end{itemize}
\end{lemma}
\begin{proof}
  Each estimate in (a) follows from the chain rule; further details are omitted. In (b), by the fundamental theorem of calculus we observe that 
  \begin{align}\label{FTC_step}
    \Rc_f \varphi = \int_0^t \nabla \varphi_{g_t(x)} \cdot  w(x) \,  dt \, , 
  \end{align}
  where $g_t(x) := \exp_{f_A^{-1} x} (t w(x))$ and $w(x) := \exp_{f_A^{-1}x}^{-1}(f^{-1} x)$. Here, for $y \in \T^2$ and $v \in T_y \T^2 \cong \R^2$, we write $\exp_y(v) = y + v$ mod 1 for the usual exponential map. The estimate \eqref{eq:sobolevEstRemainder4} follows on taking a Fourier transform and integrating by parts, noting that the extra derivative in $\varphi$ originates from \eqref{FTC_step}.
  
\end{proof}

\subsection*{Time-one estimates}

We now turn to the core of the proof of Proposition \ref{prop:mainVarCentEstimates4}, estimates on the spectral centroid and variance of $\Lc^n_{f, \kappa} b$ for $b \in C^\infty$. We begin with a treatment of estimates for a single iterate $\Lc_{f, \kappa}$. 

\begin{lemma}[Time-one centroid estimate]\label{lem:centroidEstTimeOne4}
  Let $\varphi \in L^2$ have mean zero, $\kappa \geq 0$. Then, the following estimates hold. 
  \begin{itemize}
    \item[(a)] (Effect of $\Lc_{f_A}$) It holds that 
    \begin{align*}
      \E X_{\Lc_{f_A} \varphi} = A^{-T} \E X_{\varphi} \,. 
    \end{align*}
    \item[(a)] (Effect of $\Lc_f$)
    \begin{align}
      | \E X_{\Lc_f \varphi} - E X_{\Lc_{f_A} \varphi} | \lesssim  \left( \frac{\epsilon \| \varphi\|_{W^{3,1}}}{\| \varphi\|} \right)^2 + \frac{\epsilon \| \varphi\|_{W^{4,1}}}{\| \varphi\|} 
    \end{align}
    
    \item[(b)] (Effect of $e^{\kappa \Delta}$) 
    \begin{align}\label{eq:controlDiffusivityCentroid4}
      | \E X_{e^{\kappa \Delta} \varphi} - \E X_\varphi|  \lesssim \kappa\left(  \frac{\| \varphi\|_{H^1}^2}{\| \varphi\|^2} \frac{\|  \varphi\|_{H^{1/2}}^2}{\| \varphi\|^2} +  \frac{\| \varphi\|^2_{H^{3/2}} }{\| \varphi\|^2}\right) 
    \end{align}
    so long as $\kappa \frac{\| \varphi\|_{H^1}^2}{\| \varphi\|^2} \leq 1/10$. 
  \end{itemize}
\end{lemma}

\begin{proof}
For (a) we have $\widehat{\Lc_{f_A} \varphi}(k) = \varphi(A^T k)$, hence
\begin{align*}
  \E X_{\Lc_{f_A} \varphi} & = \frac{1}{\| \Lc_{f_A} \varphi\|^2} \sum_{k} k | \hat \varphi(A^T k)|^2 \\ 
  & = \frac{1}{\| \varphi\|^2} \sum_{k} A^{-T} k | \hat \varphi(k)|^2  = A^{-T} \E X_{\varphi} 
\end{align*}
having used the fact that $\| \Lc_{f_A} \varphi\| = \| \varphi\|$.

For part (b), we have 
  \begin{align}
    \E X_{\Lc_f \varphi} & = \frac{1}{\| \varphi\|^2} \sum_{k \in \Z^2 \setminus \{ 0 \}}  k |\widehat{\Lc_f \varphi}(k)|^2 \\ 
    & = \frac{1}{\| \varphi\|^2}\sum_{k \in \Z^2 \setminus \{ 0 \}} \left[  k |\widehat{\Lc_{f_A} \varphi}(k)|^2 + O\big(|k| | \widehat{\Rc_f \varphi}(k)|^2 + |k| |\widehat{\Lc_{f_A} \varphi(k)}| |\widehat{\Rc_f \varphi(k)}| \big)\right]
  \end{align}
  The first term inside the $[\cdots]$ sums precisely to $\E X_{\Lc_{f_A}}$. 
  For the remaining terms, the Sobolev estimates for $\Rc_f \varphi$ from Lemma \ref{lem:sobolevEst4}(b) give
  \begin{align}
     \| \varphi \|^2 |\E X_{\Lc_f \varphi} - \E X_{\Lc_{f_A} \varphi}| & \lesssim  \sum_k |k| \cdot \frac{\epsilon^2 \| \varphi\|_{W^{3,1}}^2}{|k|^4} + |k| |\widehat{\Lc_{f_A} \varphi}(k)| \frac{\epsilon \| \varphi\|_{W^{4,1}}}{|k|^3} \\ 
     & \lesssim \epsilon^2 \| \varphi\|_{W^{3,1}}^2 + \epsilon \| \varphi\|_{W^{4,1}} \sum_k |k|^{-2} |\widehat{\Lc_{f_A} \varphi}(k)| \\ 
     & \leq \epsilon^2 \| \varphi\|_{W^{3,1}}^2 + \epsilon \| \varphi\|_{W^{4,1}} \left(\sum_k |k|^{-4}\right)^{1/2} \left( \sum_k |\widehat{\Lc_{f_A} \varphi}(k)|^2 \right)^{1/2} \\ 
     & \lesssim \epsilon^2 \| \varphi\|_{W^{3,1}}^2 + \epsilon \| \varphi\|_{W^{4,1}} \| \varphi\|
  \end{align} 

For part (c), we compute 
  \begin{align*}
    \E X_{e^{\kappa \Delta} \varphi} - \E X_\varphi =  \sum_{k \in \Z^2 \setminus \{ 0 \}} & \bigg[ \underbrace{\left( \frac{1}{\| e^{\kappa \Delta} \varphi \|^2 } - \frac{1}{\| \varphi\|^2} \right) k e^{- 2 \kappa |k|^2} | \hat \varphi(k)|^2}_{(I)} \\ 
    & + \underbrace{\frac{1}{\| \varphi\|^2} k (e^{- 2 \kappa |k|^2} - 1) | \hat \varphi(k)|^2}_{(II)} \bigg] \,. 
  \end{align*}
  For each $k \in \Z^2 \setminus \{ 0 \}$, term (I) bounded in absolute value from above by 
  \begin{align}
     {\| \varphi\|^2 - \| e^{\kappa \Delta} \varphi \|^2 \over \| \varphi\|^2 \| e^{\kappa \Delta} \varphi \|^2} \cdot |k| |\hat \varphi(k)|^2 \,. 
  \end{align}
  Following energy estimate in first lines of the proof of Lemma \ref{lem:enhancedDissipLower2}, we get 
  \begin{align}\label{eq:energyEst4}0 \leq \| \varphi \|^2 - \| e^{\kappa \Delta} \varphi\|^2 \leq \kappa \| \varphi\|_{H^1}^2.\end{align}
  Therefore, when $\kappa \frac{\| \varphi\|_{H^1}^2}{\| \varphi\|^2} \leq \frac{1}{10}$, (I) is bounded by 
  \begin{align}
    \lesssim \frac{1}{\| \varphi \|^2} \cdot \frac{\kappa \| \varphi\|^2_{H^1}}{\| \varphi \|^2} \cdot |k| |\hat \varphi(k)|^2 \, . 
  \end{align}
  Summing over $k$ gives the first term on the RHS of \eqref{eq:controlDiffusivityCentroid4}. 

  Term (II) is controlled 
  \begin{align}
    \leq \frac{\kappa}{\| \varphi\|^2} \cdot \left| \frac{e^{-2 \kappa |k|^2} - 1}{\kappa |k|^2}\right| |k|^3 |\hat \varphi(k)|^2 \,, 
  \end{align}
  giving the second term on the RHS of \eqref{eq:controlDiffusivityCentroid4}. 
\end{proof}

\begin{lemma}[Time-one variance estimate]\label{lem:varianceTimeOne4}
  Let $\varphi \in L^2$ have mean zero, $\kappa \geq 0$. Then, the following estimates hold. 
  \begin{itemize}
    \item[(a)] (Effect of $\Lc_{f_A}$) It holds that 
    \[\Var(X_{\Lc_{f_A} \varphi}) \leq |A|^2 \Var(X_{\varphi}) \,.  \]
    \item[(b)] (Effect of $\Lc_f$)
    \begin{align}
      \Var(X_{\Lc_f \varphi})  \leq & \Var(X_{\Lc_{f_A} \varphi}) + O \bigg[ \epsilon^2 \frac{\| \varphi\|_{W^{4,1}}^2}{\| \varphi\|^2} + \epsilon \frac{\| \varphi\|_{W^{5,1}}}{\| \varphi\|} \\ 
      & + \epsilon \left( 1 +  \left(\frac{\| \varphi\|_{H^{1/2}}}{\| \varphi\|}\right)^4 \right) \cdot \left( 1 + \epsilon \frac{\| \varphi\|_{W^{3,1}}^2}{\| \varphi\|^2} \right)  \bigg] \,. 
    \end{align}
    \item[(c)] (Effect of $e^{\kappa \Delta}$) 
  \begin{align} \label{eq:timeOneVarLine1}
    \Var(X_{e^{\kappa \Delta} \varphi}) & \leq \left[ 1 + O \left( \kappa {\| \varphi\|_{H^1}^2 \over \| \varphi\|^2} \right) \right] \Var(X_\varphi) \\ 
    & \leq \Var(X_\varphi) + O \left(\kappa {\| \varphi\|_{H^1}^4 \over \| \varphi\|^4} \right) \label{eq:timeOneVarLine2}
  \end{align}
  provided $\kappa \frac{\| \varphi\|^2_{H^1}}{\| \varphi\|^2} \leq 1/10$.
  \end{itemize}
\end{lemma}

\begin{proof}
For (a), we compute 
\begin{align*}
   \Var(X_{\Lc_{f_A} \varphi}) & = \frac{1}{\| \Lc_{f_A} \varphi \|^2} \sum_k |k - \E(X_{\Lc_{f_A} \varphi}) |^2 |\hat \varphi(A^T k)|^2 \\ 
  & = \frac{1}{\| \varphi \|^2} \sum_k \left|A^{-T} \big( k -  \E(X_{\varphi}) \big) \right|^2 |\hat \varphi(k)|^2 \\ 
  & \leq |A^{-T}|^2 \Var(X_\varphi) = |A|^2 \Var(X_{\varphi})
\end{align*}
noting $|A^{-T}| = |A^T| = |A|$ since $\det A = 1$ and $A$ is a  $2 \times 2 $ matrix. 

  For part (b): using Lemma \ref{lem:basicSpectralProps4}(i) we estimate

  \begin{align}
    \| \Lc_f \varphi\|^2 \Var(X_{\Lc_f \varphi}) \leq \sum_k |k - \E X_{\Lc_{f_A} \varphi}|^2   \bigg( |\widehat{\Lc_{f_A \varphi}}(k)|^2  + |\widehat{\Rc_f \varphi}(k)|^2 + 2 |\widehat{\Rc_f \varphi}(k)||\widehat{\Lc_{f_A} \varphi}(k)| \bigg) \,. 
  \end{align}
  The first parenthetical term on the RHS sums to give $ \| \varphi \|^2 \Var(X_{\Lc_{f_A} \varphi})$. The remaining terms are estimated similarly to the spectral centroid terms in Lemma \ref{lem:centroidEstTimeOne4}. We use liberally here the easily-checked, general estimate 
  \begin{align}\label{eq:centroidEst4}
    | \E(X_\varphi) | \lesssim \frac{\| \varphi\|_{H^{1/2}}^2 }{\| \varphi \|^2} 
  \end{align}
  for $\varphi \in H^{1/2}$ of mean zero. 

  For (c) we similarly have
  \begin{align}
    \Var(X_{e^{\kappa \Delta} \varphi}) & \leq \frac{1}{\| e^{\kappa \Delta} \varphi\|^2} \sum_k |k - \E(X_{\varphi})|^2 e^{- 2 \kappa |k|^2} |\hat \varphi(k)|^2 \\ 
    & \leq \frac{\| \varphi\|^2}{\| e^{\kappa \Delta} \varphi\|^2} \Var(X_{\varphi}) \, . 
  \end{align}
  Using \eqref{eq:energyEst4} as in the proof of Lemma \ref{lem:centroidEstTimeOne4} completes the proof of \eqref{eq:timeOneVarLine1}. The bound in \eqref{eq:timeOneVarLine2} follows from \eqref{eq:centroidEst4} and the additional estimate 
  \begin{align} \label{eq:varEstimate4} \Var(X_\varphi) \lesssim \frac{\| \varphi\|_{H^{1}}^2 }{\| \varphi \|^2} \, , \end{align}
  for mean-zero $\varphi \in H^1$. 
\end{proof}

\subsection*{Iterated estimates}

{\bf Notation. } We conveniently summarize the results of the time-one estimates of Lemmas \ref{lem:centroidEstTimeOne4}, \ref{lem:varianceTimeOne4} as follows. 
With $\epsilon, \kappa$ as above, let us write $\eta = \max\{ \epsilon, \kappa\}$. Given $\varphi \in C^\infty$ of mean zero, the estimates of Lemmas \ref{lem:centroidEstTimeOne4}, \ref{lem:varianceTimeOne4} indicate that 
\begin{gather*} | \E X_{\Lcfk \varphi} - \E X_{\varphi}|  \lesssim \eta \mathcal{E}[\varphi] \, , \\ 
  \Var(X_{\Lcfk \varphi})  \leq |A|^2 \Var(X_\varphi) + O(\eta \mathcal{E}[\varphi]) \, 
\end{gather*}
where
\begin{gather*}
  \mathcal{E}[\varphi] = \max \bigg\{ \bigg(\frac{\| \varphi\|_{W^{3,1}}}{\| \varphi\|}\bigg)^2 , \frac{\| \varphi\|_{W^{4,1}}}{\| \varphi\|} , \left(\frac{\| \varphi\|_{H^1}}{\| \varphi\|} \frac{\| \varphi\|_{H^{1/2}}}{\| \varphi\|}\right)^2, \left(\frac{\|\varphi\|_{H^{3/2}}}{\| \varphi\|}\right)^2 , \\ 
  \bigg(\frac{\| \varphi\|_{W^{4,1}}}{\| \varphi\|}\bigg)^2 , \frac{\| \varphi\|_{W^{5,1}}}{\| \varphi\|} , \left(\frac{\| \varphi\|_{H^1}}{\| \varphi\|} \right)^2 , \left(\frac{\| \varphi\|_{H^1}}{\| \varphi\|} \right)^4 \\ 
  \left( 1 +  \left(\frac{\| \varphi\|_{H^{1/2}}}{\| \varphi\|}\right)^4 \right) \cdot \left( 1 + \frac{\| \varphi\|_{W^{3,1}}^2}{\| \varphi\|^2} \right) 
  \bigg\} \, , 
\end{gather*}
provided that $\eta \mathcal{E}[\varphi] \leq 1/10$. 

Let $\varphi_0 := b$ and, for $n \geq 1$, let $\varphi_n = \Lc^n_{f, \kappa} b = \Lcfk \varphi_{n-1}$. Let 
\[X_n = X_{\varphi_n} \, , \qquad \Ec_n = \Ec[\varphi_n] \, . \]
Finally, with $b = e_{k_0}, k_0 \in \Z^2$ fixed, define $k_n = (A^{-T}) k_0$. 
To complete the proof of Proposition \ref{prop:mainVarCentEstimates4}, it suffices to prove the following. 

\begin{lemma}\label{lem:estErrorVarianceCentroid4}
  Let $\kappa > 0$. There exists $M > 1$, depending only on $\| f \|_{C^5}$, with the property that 
\begin{align*}
  \Ec_n \lesssim M^n
\end{align*}
for all $n \lesssim |\log \kappa| + 1$. 
\end{lemma}
We note that $\| f \|_{C^5}$ is controlled by $d_{C^\infty}(f, f_A)\leq \epsilon$, so as long as $\epsilon \leq 1$ we have that the value $M$ is uniform over all $f$ we consider here, and depends only on the CAT map $f_A$. 
\begin{proof}

  To start, using Lemma \ref{lem:sobolevEst4}(a) we can fix $C_0, M > 1$, depending only on $\| f\|_{C^5}$, for which 
  \begin{align}\label{eq:higherRegGrowth4}
    \| \varphi_n\|_{\ast} \leq C_0 M^{n/2}
  \end{align} 
  for all $n$, and for each of $\ast = H^2, W^{5,1}$. 
  
  It remains to control the denominator terms involving $\| \varphi\|$ from below. For this, Equation \eqref{eq:energyEst4} implies 
  \begin{align*}
    \| \varphi_n\|^2 &\geq \| \varphi_{n-1}\|^2 - \kappa \| \varphi_{n-1}\|^2_{H^1} \geq \cdots \\ 
    & \geq 1 - \kappa \sum_0^{n-1} \| \varphi_\ell\|_{H^1}^2
  \end{align*}
  noting that $\| \varphi_0\|^2 = \| e_{k_0}\|^2 = 1$. Plugging in \eqref{eq:higherRegGrowth4}, we see that $\kappa \sum_0^{n-1} \| \varphi_\ell\|_{H^1}^2 < 1/2$ when $n \lesssim_M |\log \kappa| + 1$. 
  Using the general inequality $\frac{1}{1-a} \leq 1 + 2 a$ for $a \in [0,1/2]$ completes the proof. 
\end{proof}

Turning finally to the iterated estimates: for the time-$n$ centroid $\E X_n$ we have
\begin{align*}
  | \E X_n - k_n | & \leq \sum_{\ell = 0}^{n-1} \big| (A^{-T})^{n - 1 - \ell} \E X_{\ell + 1} - (A^{-T})^{n-\ell} \E X_{\ell} \big| \\ 
  & \leq \sum_{\ell = 0}^{n-1} |A|^{n-1-\ell} \big|  \E X_{\ell + 1} - A^{-T} \E X_{\ell} \big| \\ 
  & \lesssim \sum_{\ell = 0}^{n-1} |A|^{n-1-\ell} \cdot \eta \Ec_\ell \lesssim \eta |A|^{n-1} M^{n-1} \,, 
\end{align*}
using Lemma \ref{lem:estErrorVarianceCentroid4} in the last line. For the iterated variance, 
\begin{align*}
  \Var(X_n) & \leq |A|^2 \Var(X_{n-1}) + \eta \Ec_{n-1} \\ 
  & \leq (|A|^2)^2 \Var(X_{n-2}) + |A|^2 \eta \Ec_{n-2} + \eta \Ec_{n-1} \leq \cdots \\ 
  & \leq (|A|^n)^2 \Var(X_0) + \eta \sum_{\ell = 1}^{n-1} (|A|^2)^{\ell-1} \Ec_{n-\ell} \\ 
  & = \eta \sum_{\ell = 1}^{n-1} (|A|^2)^ {\ell-1} \Ec_{n-\ell} \\
  & \lesssim \eta |A|^{2n} M^n \,. 
\end{align*}
where above we use that $\Var(X_0) = \Var(X_{e_{k_0}}) = 0$, since $e_{k_0}$ is a pure Fourier mode with zero variance.
Enlarging $M$ completes the proof.


\begin{remark}[Origin of the critical wavenumber $R_{\rm crit}$] \label{rmk:criticalTimescale4}
  We give an account here of the origins of the critical wavenumber $R_{\rm crit}$ and corresponding shell number $\ell_{\rm crit}$ appearing in Theorems \ref{thm:pulseLocalization4}, \ref{thm:expShellLaw4} respectively. 

To start, Theorem \ref{thm:cumLaw2}(a) exhibits the statistically stationary scalar $g_{f, \kappa}$ as a superposition of time-$n$ `pulses' $\varphi_n = \Lcfk^n b$. In the proof of Theorem \ref{thm:pulseLocalization4}, it was demonstrated that $\Var(X_{\varphi_n})$ is small for times $n \leq N_{\rm crit}$, a critical timescale past which estimates break down and $\Var(X_{\varphi_n})$ may be large. When $\kappa \ll \epsilon$, the effects of diffusivity are outweighed by the effects of the perturbation $f$ of $f_A$, and the critical timescale $N_{\rm crit} \approx |\log \epsilon|$ is, roughly, the time horizon past which individual trajectories of $f^n (x)$ significantly deviate from those of $f^n_A(x)$. 

As we saw, the critical wavenumber $R_{\rm crit}$ was set so that $\log R_{\rm crit} \approx N_{\rm crit}$, and chosen specifically so that the bulk of Fourier mass from the `tamed' pulses $\varphi_n, n \leq N_{\rm crit}$ lies below $R_{\rm crit}$. Finally, critical exponential shell $\ell_{\rm crit}$ chosen so that $L^{\ell_{\rm crit}} \approx R_{\rm crit}$, which allowed us to estimate Fourier mass in the shells $[L^\ell, L^{\ell + 1}], \ell < \ell_{\rm crit}$ from only `tame' pulses. 
\end{remark}

\appendix




\section{Proof of the uniform Lasota-Yorke estimate (Proposition \ref{prop:uniformLasotaYorke3})}\label{sec:APP} \label{Jezequel}

The setup is as follows. We let $A$ be a $2 \times 2$ hyperbolic matrix with determinant $\pm 1$ and integer entries, with $f_A : \T^2 \to \T^2$ the associated CAT map as in Section \ref{sec:intro}. The cones $\Cc^u, \Cc^s$ and the anisotropic norm $\| \cdot \|_{\Hc^p}, p > 0$ are as defined in Section \ref{subsec:catMap3}. 
Our aim is to prove Proposition \ref{prop:uniformLasotaYorke3}, reproduced here for convenience and in a slightly more general form (c.f. Remark \ref{rmk:unifLYcomments3}). Below, we abuse notation somewhat and write $\| f\|_{C^\infty} = d_{C^\infty}(f, \Id)$. 

\begin{proposition}[Uniform Lasota-Yorke inequality]
    Let $p > 0, L > p, \nu \in (1, \lambda)$ and $\hat C > 0$. Then, there exist $\epsilon, C_2, C_3, M > 0$ such that the following holds. 
    Assume $\kappa \geq 0$ and that $f$ is a volume-preserving diffeomorphism of $\T^d$ such that $d_{C^1}(f, f_A) < \epsilon$ and $\| f \|_{C^\infty} \leq \hat C$. Then, for all $n \geq 1$ and $\varphi \in \Hc^p_0$, it holds that 
    \begin{align}\label{eq:LYapp}\| \Lc_{f, \kappa}^n \varphi \|_{\Hc^p} \leq C_2 \nu^{-np} \| \varphi \|_{\Hc^p} + C_3 M^n \| \varphi\|_{H^{-L}} \,.  \end{align}
\end{proposition}

The proof we give largely follows \cite{baladi2007anisotropic} (see also \cite{baladi2017quest}). We are indebted to the unpublished notes of J\'ez\'equel \cite{jezequel2023minicourse} on this topic, from which we borrow significantly in all but Section \ref{subsec:treatDiffusivityAPP}. 

\subsubsection*{Plan for Section \ref{sec:APP}}

After some preliminaries (Section \ref{subsec:prelimAPP}), we give the proof of the uniform Lasota-Yorke inequality when $\kappa = 0$ in Section \ref{subsec:proofLYapp}. A key step in the proof, Lemma \ref{lem:nonstatPhaseApp} is proved separately in Section \ref{subsec:proofStatPhaseApp}. Finally, changes needed in the treatment of $\kappa > 0$ are documented in Section \ref{subsec:treatDiffusivityAPP}.

\subsection{Preliminaries}\label{subsec:prelimAPP}

\subsubsection*{Standing assumptions}

Throughout, $\nu \in (1, \lambda)$ is fixed, as are $p > 0$ and $L > p$. A volume-preserving diffeomorphism $f$ near to $f_A$ is fixed. At times, we will write $\Lc = \Lc_f$ for brevity.

It will be convenient to write \[F = f^{-1}, F_A = f^{-1}_A \, , \] noting that estimates on $d_{C^1}(f, f_A)$ and $\| f \|_{C^\infty}$ translate to those of $d_{C^1}(F, F_A), \| F\|_{C^\infty}$, and noting that $\Lc \varphi = \varphi \circ F$.

Next, we note that $F, F_A$ exchange the roles of stable/unstable directions due to time reversal. As such, it will be convenient to use the notation 
\[\Cc^s_* = \Cc^u \,,  \qquad \Cc^u_* = \Cc^s \, , \]
noting that $\Cc^u_*$ is unstable and $\Cc^s_*$ stable, respectively, for $A^{-T}$. 

Next, $\epsilon > 0$ will be assumed sufficiently small to ensure hyperbolicity of $f_A$ translates to that of $f$, hence that of $F = f^{-1}$. Precisely, we shall assume throughout that there exists 
$C_0 > 0$ such that for all $x \in \T^2, n \geq 1$, 
        \begin{gather}\label{eq:expansionFapp}
            \begin{gathered}
                (D_x F^n)^T(\Cc^u_*)  \subset \Cc^u_* \, , \quad \text{ and} \\ 
            \left|\left(D_{x}F^{n}\right)^{T}w\right|  \geq C_{0}^{-1}\nu^{n}\left|w\right|\quad\forall w\in \mathcal{C}_*^{u} \,, 
            \end{gathered}
        \end{gather}
        while 
        \begin{gather}\label{eq:contractionFapp}\begin{gathered}
            (D_x F^n)^T(\Cc^s_*) \supset \Cc^s_* \, , \quad \text{ and} \\ 
            \left|\left(D_{x}F^{n}\right)^{T}v\right| \leq C_{0}\nu^{-n}\left|v\right|\quad\forall v\in (D_x F^n)^{-T} \mathcal{C}_*^{s}  \, . 
        \end{gathered} \end{gather}
We emphasize that these conditions are $C^1$ open in the mapping $f$, hence $C^1$ open in $F$. We also point out that \eqref{eq:expansionFapp}, \eqref{eq:contractionFapp} are all we shall require of the $C^1$ closeness of $f, f_A$, a fact relevant in Section \ref{subsec:treatDiffusivityAPP}. 

Finally, $\epsilon > 0$ will be taken small enough so that uniformly over $v \in \Cc^u_* \setminus \{ 0 \}$ and $w \in \Cc^s_* \setminus \{ 0 \}$, one has 
\begin{align} \label{eq:angleBoundAPP}\angle \left((D_x F^n)^T v, w \right) \geq \vartheta > 0\end{align}
for some positive $\vartheta$ and for all $n$; that this can be arranged follows from the fact that $A^{-T}$ maps $\Cc^u_*\setminus \{ 0 \}$ strictly into its interior. 


\subsection*{Alternative form of $\| \cdot \|_{\Hc^p}$}

For $\varphi \in C^\infty$ and $N \geq 1$, $t \in \{ u, s\}$, we define the Littlewood-Paley projections 
    \begin{align*}
        \Pi_{N}^{t}\varphi & :=\sum_{\substack{k\in \mathcal{C}_*^{t} \, , \\
        \left|k\right|\in[2^{N-1},2^{N})
        }
        }\widehat{\varphi}\left(k\right)e_{k}
        \end{align*}
    and  
    \[ \| \varphi \|_{\Hc^p}' := \left|\widehat{\varphi}\left(0\right)\right|^{2}+\sum_{N\geq1}\left(2^{-2p N}\left\Vert \Pi_{N}^{u}\varphi\right\Vert _{L^{2}}^{2}+2^{2p N}\left\Vert \Pi_{N}^{s}\varphi\right\Vert _{L^{2}}^{2}\right)
    \]
    It is straightforward to check that $\| \cdot \|_{\Hc^p}$ and $\| \cdot \|_{\Hc^p}'$ are uniformly equivalent up to a factor depending only on $p$. As a reminder, we point out that $\Pi^u_N$ projects to modes in $\Cc^u_* = \Cc^s$ and corresponds to negative-regularity Fourier modes, while $\Pi^s_N$ projects to modes in $\Cc^s_* = \Cc^u$ and corresponds to positive-regularity modes. 


\subsubsection*{Reduction}

We now argue that to prove \eqref{eq:LYapp} for $\kappa = 0$, it  suffices to show the following: if $F = f^{-1}$ satisfies the hyperbolicity estimates \eqref{eq:expansionFapp}, \eqref{eq:contractionFapp}, then there exists a constant $\tilde C > 0$ such that for any $n \geq 1$, there exists $C_n > 0$ such that 

\begin{align}\label{eq:LYsuffAPP}
    \left\Vert \mathcal{L}^{n}\varphi\right\Vert _{\mathcal{H}^p}'\leq \tilde C \nu^{-p n}\left\Vert \varphi\right\Vert _{\mathcal{H}^p}'+C_{n}\left\Vert \varphi\right\Vert _{H^{-L}}
\end{align}

Crucially, while $C_n$ is permitted to depend on $n$, the constant  $\tilde C$ does not, and both constants are uniform in the upper bound $\hat C$ for $\| f \|_{C^\infty}$. With $n_0$ fixed so that $\tilde C \nu^{-p n_0} < 1$, it will follow that 
\[\left\Vert \mathcal{L}^{n}\varphi\right\Vert _{\mathcal{H}^p} ' \leq ( 1 + \| \Lc_f \|_{\Hc^p}')^{n_0} (\tilde C \nu^{-p n_0})^q \left\Vert \varphi\right\Vert _{\mathcal{H}^p}'+  C_{n_0} ( 1 + \| \Lc_f \|_{H^{-L}}')^{n}  \left\Vert \varphi\right\Vert _{H^{-L}}\]
where $n = q n_0 + r, r \in \{ 0, \dots, n_0 - 1\}$. Equation \eqref{eq:LYapp} now follows, on replacing $\nu$ above with $\nu' \in (\nu, \lambda)$ and taking $n_0$ yet larger so that $C (\nu')^{-n_0 p} \leq \nu^{-n_0 p}$.

\subsection{Proof of time-$n$ Lasota-Yorke estimate \eqref{eq:LYsuffAPP}}\label{subsec:proofLYapp}

Let
\[
\Gamma=\mathbb{N}\times\left\{ u,s\right\} =\left\{ \left(N,t\right):N\in\mathbb{N},t\in\left\{ u,s\right\} \right\} 
\]
and, in a slight abuse of notation, we shall regard each $\Gamma_{N, t}$ as a subset of $\Z^2$, with 
\[
k\in\Gamma_{N,t}\iff\left|k\right|\in[2^{N-1},2^{N})\text{ and }k\in\mathcal{C}_*^{t} \,, \quad t \in \{ u, s\} \, .
\]

Below, we will record `significant' time-$n$ transitions of Fourier mass between the sets $\Gamma_{N,t}$ using the relation $\hookleftarrow_{n}$ on $\Gamma$. Given $(N, t), (N', t') \in \Gamma$, we will write $(N, t) \hookleftarrow_{n} (N', t')$ and say \emph{$(N', t')$ leads to $(N, t)$} if one of the following (mutually exclusive) conditions is met: 
\begin{enumerate}
\item $t = t' = u$ and $N\geq N'+n\log_{2}\nu-\log_{2}\left(\frac{C_{0}}{4}\right)$; 
\item $t = t' = s$ and $N'\geq N+n\log_{2}\nu-\log_{2}\left(\frac{C_{0}}{4}\right)$; or 
\item $t = u, t' = s$ and $\max\left(N,N'\right)\geq n\log_{2}\nu$. 
\end{enumerate}
If none of the conditions are met, we write $(N, t) \cancel{\hookleftarrow}_{n} (N', t')$. 

\begin{remark}
These conditions are chosen to emulate the way $A^{-T}$ transmits Fourier modes. Rule 1 captures the way that $A^{-T}$ leaves invariant $\Cc^u_*$ and expands vectors along it (c.f. \eqref{eq:expansionFapp}), while rule 2 captures that $A^{-T}$ contracts vectors $v$ for which $v, A^{-T} v \in \Cc^s_*$ (c.f. \eqref{eq:contractionFapp}). Finally, rule 3 allows for transitions from the stable cone $\Cc^s_*$ to $\Cc^u_*$. 
\end{remark}

The following estimate controls the `insignificant' transitions $(N, t) \cancel{\hookleftarrow}_{n} (N', t')$. 

    

\begin{lemma}[Non-stationary phase]\label{lem:nonstatPhaseApp}
Assume the mapping $F$ satisfies \eqref{eq:expansionFapp} and \eqref{eq:contractionFapp}. Let $\beta > 0, n \geq 1$. Then, there exists a constant $C_{n, \beta} > 0$ such that 
if $\left(N,t\right)\cancel{\hookleftarrow}_{n}\left(N',t'\right)$, then 
we have 
\begin{equation}
\left\Vert \Pi_{N}^{t}\mathcal{L}^{n}\Pi_{N'}^{t'}\right\Vert _{L^{2}\to L^{2}}\leq C_{n,\beta}2^{-\beta\max\left(N,N'\right)}\label{eq:non-stationary}
\end{equation}
\end{lemma}

The proof is deferred till Section \ref{subsec:proofStatPhaseApp}. For the rest of Section \ref{subsec:proofLYapp} we assume Lemma \ref{lem:nonstatPhaseApp} and complete the proof of the estimate \eqref{eq:LYsuffAPP}.


Let $\varphi \in \Hc^p_0, n \geq 1, L > p$. Below, $\beta > L$ will be enlarged a finite number of times to suit our purposes. Towards the end of estimating $\| \Lc^n \varphi\|_{\Hc^p}'$, we begin by estimating the stable contribution  $\Pi^s_N \Lc^n \varphi$ for $N \geq 1$ is fixed. 
\begin{align*}
 & 2^{p N}\left\Vert \Pi_{N}^{s}\mathcal{L}^{n}\varphi\right\Vert _{L^{2}}\\
 & \leq2^{p N}\left\Vert \Pi_{N}^{s}\mathcal{L}^{n}\sum_{N'\geq1}\Pi_{N'}^{s}\varphi\right\Vert _{L^{2}}+2^{p N}\left\Vert \Pi_{N}^{s}\mathcal{L}^{n}\sum_{N'\geq1}\Pi_{N'}^{u}\varphi\right\Vert _{L^{2}}\\
 & \leq2^{p N}\left\Vert \Pi_{N}^{s}\mathcal{L}^{n}\sum_{N':\left(N,s\right)\hookleftarrow_{n}\left(N',s\right)}\Pi_{N'}^{s}\varphi\right\Vert _{L^{2}}+2^{p N}\sum_{\left(N',t'\right):\left(N,s\right)\cancel{\hookleftarrow}_{n}\left(N',t'\right)}\left\Vert \Pi_{N}^{s}\mathcal{L}^{n}\Pi_{N'}^{t'}\varphi\right\Vert _{L^{2}} \,. 
\end{align*}
Above, we are using the fact that $(N, s) \hookleftarrow_n (N', t')$ implies $t' = s$, i.e., wavenumbers in $\Cc^u_*$ never lead to those in $\Cc^s_*$. Using \eqref{eq:non-stationary} to control the second summation on the RHS, and using that $\|\Pi^s_N \Lc^n\|_{L^2} \leq 1$ due to volume-preservation, gives that $2^{p N}\left\Vert \Pi_{N}^{s}\mathcal{L}^{n}\varphi\right\Vert _{L^{2}}$ is bounded 
\begin{align*}
    & \leq2^{p N}\left\Vert \mathcal{L}^{n}\sum_{N':\left(N,s\right)\hookleftarrow_{n}\left(N',s\right)}\Pi_{N'}^{s}\varphi\right\Vert _{L^{2}}+2^{p N}\sum_{N'\geq1}O_{n,\beta}\left(2^{-\beta\max\left(N,N'\right)}\right)2^{N'L}\left\Vert \Pi_{N'}\varphi\right\Vert _{H^{-L}}\\
 & \leq2^{p N}\left\Vert \sum_{N':\left(N,s\right)\hookleftarrow_{n}\left(N',s\right)}\Pi_{N'}^{s}\varphi\right\Vert _{L^{2}}+O_{n,p,\beta,L}\left(2^{\left(p+ L -\beta\right)N}\right)\left\Vert \varphi\right\Vert _{H^{-L}} \, , 
\end{align*}
hence 
\begin{align} \label{eq:stableEst12app} 2^{2 p N} \| \Pi^s_N \Lc^n \varphi \|_{L^2}^2 \leq 2^{2 p N + 1} \left\Vert \sum_{N':\left(N,s\right)\hookleftarrow_{n}\left(N',s\right)}\Pi_{N'}^{s}\varphi\right\Vert _{L^{2}}^2 + O_{n, p, \beta, L} (2^{2 (p + L - \beta)N} ) \| \varphi\|_{H^{-L}}^2 \end{align}
on squaring and using the general inequality $(a + b)^2 \leq 2 (a^2 + b^2)$ for $a, b \geq 0$. 
Taking $\beta > p + |L|$, summing the second RHS term of \eqref{eq:stableEst12app} over $N \geq 1$ gives a quantity bounded as $\lesssim_{p, n, \beta, L} \| \varphi\|_{H^{-L}}^2$. For the first term, the fact that $\{ \Pi^s_{N'} \varphi : N' \geq 1\}$ is orthogonal gives
\begin{align*}
   &  \sum_{N \geq 1} 2^{2 p N+1}\left\Vert \sum_{N':\left(N,s\right)\hookleftarrow_{n}\left(N',s\right)}\Pi_{N'}^{s}\varphi\right\Vert _{L^{2}}^2 = \sum_{N \geq 1} 2^{2 p N} \sum_{N':\left(N,s\right)\hookleftarrow_{n}\left(N',s\right)} \| \Pi_{N'}^{s}\varphi\|_{L^2}^2 \\
   & = \sum_{N\geq1} 2^{2p N+1} \sum_{N'\geq N+n\log_{2}\nu-\log_{2}\left(\frac{C_{0}}{4}\right)} \| \Pi^s_{N'} \varphi\|_{L^2}^2 \\ 
   & = \sum_{N' \geq 1} \sum_{N \leq N' - n \log \nu + \log\left( \frac{C_0}{4}\right)} 2^{2p N+1} \| \Pi^s_{N'} \varphi\|_{L^2}^2 \\ 
   & \lesssim_{\neg n} \sum_{N' \geq 1} 2^{2 p N'} \nu^{-2 p n} \| \Pi^s_{N'} \varphi\|_{L^2}^2 \lesssim_{\neg n} \nu^{- 2 p n} \| \varphi\|_{\Hc^p}^2 \,, 
\end{align*}
having applied Tonelli from the second to the third lines. In total, for the stable contribution of $\Lc^n \varphi$ we have 
\begin{align}
    \sum_N 2^{2 p N} \| \Pi^s_N \Lc^n \varphi\|_{L^2}^2 \leq O_{\neg n} (\nu^{- 2 p n}) (\| \varphi\|_{\Hc^p}')^2 + O_{n, p, \beta, L} (\| \varphi\|_{H^{-L}}^2) \,. 
\end{align}

Turning to the unstable contribution $\Pi^u_N \Lc^n \varphi$, we estimate 
\begin{align*}
 & 2^{-p N}\left\Vert \Pi_{N}^{u}\mathcal{L}^{n}\varphi\right\Vert _{L^{2}}\\
 & \leq2^{-p N}\left\Vert \Pi_{N}^{u}\mathcal{L}^{n}\sum_{\substack{(N', t') : \\ (N, u) \hookleftarrow_n (N', t')}}\Pi_{N'}^{t'}\varphi\right\Vert _{L^{2}}+\sum_{\substack{(N', t') : \\ (N, u) \cancel\hookleftarrow_n (N', t')}}2^{-p N}\left\Vert \Pi_{N}^{u}\mathcal{L}^{n}\Pi_{N'}^{s}\varphi\right\Vert _{L^{2}}\\
 & \leq2^{-p N}\left\Vert \sum_{N':\left(N,u\right)\hookleftarrow_{n}\left(N',u\right)}\Pi_{N'}^{u}\varphi\right\Vert _{L^{2}}+2^{-p N}\left\Vert \sum_{N':\left(N,u\right)\hookleftarrow_{n}\left(N',s\right)}\Pi_{N'}^{s}\varphi\right\Vert _{L^{2}} \\ & \quad +O_{n,\beta,L}\left(2^{\left(-p+ L-\beta\right)N}\right)\left\Vert \varphi\right\Vert _{H^{-L}} \, , 
\end{align*}
with the main difference now being that we incorporate transitions from $\Cc^s_*$ and $\Cc^u_*$ in the `significant' terms. Squaring this inequality and using the general fact that $(a + b + c)^3 \leq 3 (a^2 + b^2 + c^2)$ for nonnegative reals $a, b, c$ gives
\begin{align}\label{eq:intermUestApp12}\begin{aligned}
   &  2^{-2p N}\left\Vert \Pi_{N}^{u}\mathcal{L}^{n}\varphi\right\Vert _{L^{2}}^2 \\ 
   &  \leq 3 \cdot 2^{-2 p N} \left\Vert \sum_{N':\left(N,u\right)\hookleftarrow_{n}\left(N',u\right)}\Pi_{N'}^{u}\varphi\right\Vert _{L^{2}}^2 + 
    3 \cdot 2^{- 2 p N} \left\Vert \sum_{N':\left(N,u\right)\hookleftarrow_{n}\left(N',s\right)}\Pi_{N'}^{s}\varphi\right\Vert _{L^{2}}^2 \\ 
    & + O_{n, \beta, L} (2^{2 (-p + L - \beta) N}) \| \varphi\|_{H^{-L}}^2 \,. 
\end{aligned}\end{align}

Summing now in $N$, the third RHS term of \eqref{eq:intermUestApp12} is $\lesssim_{p, n , \beta, L} \| \varphi\|_{H^{-L}}^2$.
Again using $L^{2}$-orthogonality of $\left(\Pi_{N'}^{t'}\varphi\right)_{\left(N',t'\right)}$, the sum in $N$ of the first and second RHS terms of \eqref{eq:intermUestApp12} are bounded 
\begin{align*}
 & \leq\sum_{N\geq1}\sum_{N'\leq N-n\log_{2}\nu+\log_{2}\left(\frac{C_{0}}{4}\right)}2^{-2p N}\cdot3\left\Vert \Pi_{N'}^{u}\varphi\right\Vert _{L^{2}}^{2}+\sum_{N\geq1}\sum_{\substack{N': \\ \max\left(N,N'\right)\geq n\log_{2}\nu}}2^{-2p N}\cdot3\left\Vert \Pi_{N'}^{s}\varphi\right\Vert _{L^{2}}^{2} \\
 & \leq \sum_{N'\geq1}3 \left\Vert \Pi_{N'}^{u}\varphi\right\Vert _{L^{2}}^{2} \sum_{N\geq N'+n\log_{2}\nu-\log_{2}\left(\frac{C_{0}}{4}\right)}2^{-2p N}+\sum_{N'\geq1}3\cdot2^{2p N'}\left\Vert \Pi_{N'}^{s}\varphi\right\Vert _{L^{2}}^{2}\sum_{\substack{N: \\ \max\left(N,N'\right)\geq n\log_{2}\nu}}2^{-2p\left(N+N'\right)}\nonumber \\
 & \leq O_{\neg n} (\nu^{- 2 p n}) \left( \sum_{N'\geq1} 2^{-2p N'} \left\Vert \Pi_{N'}^{u}\varphi\right\Vert _{L^{2}}^{2}+\sum_{N'\geq1}2^{2p N'}\left\Vert \Pi_{N'}^{s}\varphi\right\Vert _{L^{2}}^{2} \right) \nonumber \\
 & = O_{\neg n}\left(\nu^{-2p n}\right)(\left\Vert \varphi\right\Vert _{\mathcal{H}^p}')^{2} \, , 
\end{align*}
hence in total, 
\begin{align*}
    \sum_{N} 2^{-2pN} \| \Pi^u_N \Lc^n \varphi\|^2_{L^2} \leq O_{\neg n}(\nu^{-2 p n}) ( \| \varphi\|'_{\Hc^p})^2 + O_{n, p, \beta, L} (\| \varphi\|_{H^{-L}}^2) \,. 
\end{align*}
Combining the stable and unstable contributions implies a bound of the form 
\[(\| \Lc^n \varphi\|'_{\Hc^p})^2 \leq O_{\neg n} (\nu^{- 2 p n} ) (\| \varphi\|_{\Hc^p}')^2 + O_{n, p, \beta, L} ( \| \varphi \|_{H^{-L}}^2) \,. \]
This implies \eqref{eq:LYsuffAPP}, as desired. 


\subsection{Proof of non-stationary phase estimate (Lemma \ref{lem:nonstatPhaseApp})}\label{subsec:proofStatPhaseApp}

Let $k, k' \in \Z^2 \setminus \{ 0\}$ and observe that the Fourier transform of $\Lc e_{k'}$ of the $k'$-Fourier mode $e_{k'}(x) := e^{2 \pi i (k' \cdot x)}$ is given by 
\[\Fc[\Lc^n e_{k'}](k) = \int_{\T^2} e^{2 \pi i (k' \cdot F^n x - k\cdot x)} dx \, . \]
Let 
\[G(x) = G_{k, k', n}(x) := k' \cdot F^n (x) - k \cdot x\]
denote the phase appearing in the above integral. 
In what follows, we will show that when $k \in \Gamma_{N, t}, k' \in \Gamma_{N', t'}$ and $(N, t) \cancel\hookleftarrow_n (N', t')$, the phase $G = G_{k, k', n}$ varies sufficiently fast to create cancellations, resulting in tight control on $\Fc[\Lc e_{k'}](k)$ in an argument reminiscent of the method of stationary phase. 

\medskip 

To start, fix $(N, t), (N', t')$ for which $(N, t) \cancel\hookleftarrow_n (N', t')$. If $\max(N, N') < n \log_2 \nu$, the desired estimate for $\| \Pi^t_N \Lc^n \Pi^{t'}_{N'}\|_{L^2}$ is clearly $O_{\hat C, n, \beta}(1)$, where we recall $\hat C$ is an upper bound for $\| f \|_{C^\infty}$.  So, without loss we will proceed under the additional assumption that $\max(N, N') \geq n \log_2 \nu$. 


Let $\psi \in L^2$ be a fixed mean-zero function. We compute 
\begin{align*}
\mathcal{L}^{n}\Pi_{N'}^{t'}\psi & =\sum_{k'\in\Gamma_{N',t'}} \mathcal{L}^{n}\left(\widehat{\psi}\left(k'\right)e^{i2\pi\left(k'\cdot x\right)}\right)=\sum_{k'\in\Gamma_{N',t'}}\widehat{\psi}\left(k'\right)e^{i2\pi\left(k'\cdot F^{n}x\right)}
\end{align*}
and estimate
\begin{align}
\left\Vert \Pi_{N}^{t}\mathcal{L}^{n}\Pi_{N'}^{t'}\psi\right\Vert _{L^{2}}^{2} & =\sum_{k\in\Gamma_{N,t}}\left|\sum_{k'\in\Gamma_{N',t'}}\widehat{\psi}\left(k'\right)\mathcal{F}\left(e^{i2\pi\left(k'\cdot F^{n}\left(\cdot\right)\right)}\right)\left(k\right)\right|^{2} \nonumber \\ 
& =\sum_{k\in\Gamma_{N,t}}\left|\sum_{k'\in\Gamma_{N',t'}}\widehat{\psi}\left(k'\right)\int_{\mathbb{T}^{2}}e^{i2\pi G\left(x\right)}\;dx\right|^{2}\nonumber \\
 & \leq \sum_{k\in\Gamma_{N,t}}  \left(\sum_{k'\in\Gamma_{N',t'}}\left|\widehat{\psi}\left(k'\right)\right|^{2}\right) \left( \sum_{k' \in \Gamma_{N', t'}} \left|\int_{\mathbb{T}^{2}}e^{i2\pi G_{k,k',n}\left(x\right)}\;dx\right|^{2} \right) \label{eq:HOlder}
\end{align}

To control the RHS we will show that the phase $G$ varies sufficiently fast, i.e., that the norm of its gradient
\[
\nabla G_{k, k', n}(x) = \left(D_{x}F^{n}\right)^{T}k'-k
\]
is controlled from below under our assumptions. 

\begin{claim}\label{cla:suffFastPhaseAPP}
    Assume $\max\{ N, N'\} \geq n \log_2 \nu$ and $\left(N,t\right)\cancel{\hookleftarrow}_{n}\left(N',t'\right)$ and let $k \in \Gamma_{N, t}, k' \in \Gamma_{N', t'}$. Then, 
\begin{equation}
\left|\nabla G_{k, k', n} (x)\right|\gtrsim_{n,\neg N,\neg N'}2^{\max\left(N,N'\right)}\label{eq:lower_grad_phi}
\end{equation}
\end{claim}

The proof of Claim \ref{cla:suffFastPhaseAPP} is deferred for the moment. To estimate the RHS of 
\eqref{eq:HOlder}, fix for now $k \in \Gamma_{N, t}, k' \in \Gamma_{N', t'}$ and write $G = G_{k, k', n}$ for short. Define the differential operator 
\[L\left(\psi\right):=\frac{1}{i2\pi}\frac{\nabla G\left(x\right)}{\left|\nabla G\left(x\right)\right|^{2}}\cdot\nabla\psi \, , \]
observing that (i) $L(e^{2 \pi i G}) = e^{2 \pi i G}$ and (ii) the formal $L^2$ adjoint $L^*$ is given by 
\[L^* (\psi) = -\frac{1}{i2\pi}\nabla \cdot \left(\frac{\nabla G\left(x\right)}{\left|\nabla G\left(x\right)\right|^{2}}\psi\right) \,,  \]
hence 
\[\left|\int_{\mathbb{T}^{2}}e^{i2\pi G\left(x\right)}\;dx\right|  =\left|\int_{\mathbb{T}^{2}}e^{i2\pi G\left(x\right)}\left(L^{*}\right)^{m}\left(1\right)\;dx\right|\]
for all $m \geq 1$. 

By induction, one can show that 
\[| (L^*)^m (1)(x)| \lesssim_m \frac{\| \nabla G\|_{C^m}^\ell}{| \nabla G(x)|^{\ell + m}} \]
for some $\ell = \ell_m$, pointwise in $x$, hence $|(L^*)^m(1)(x)| \lesssim_{m,n,\neg N, \neg N'} 2^{- m \max(N, N')}$ by \eqref{eq:lower_grad_phi}. Plugging back into \eqref{eq:HOlder} now gives

\begin{align*}
\left\Vert \Pi_{N}^{t}\mathcal{L}^{n}\Pi_{N'}^{t'}\psi\right\Vert _{L^{2}}^2  & \leq \sum_{k\in\Gamma_{N,t}}  \left(\sum_{k'\in\Gamma_{N',t'}}\left|\widehat{\psi}\left(k'\right)\right|^{2}\right) \left( \sum_{k' \in \Gamma_{N', t'}} \left|\int_{\mathbb{T}^{2}}e^{i2\pi G\left(x\right)}\left(L^{T}\right)^{m}\left(1\right)\;dx \right|^2  \right) \\
&\lesssim_{m,n,\neg N,\neg N'}\sum_{k\in\Gamma_{N,t}} 2^{2 N' -2m\max\left(N,N'\right)} \sum_{k'\in\Gamma_{N',t'}}\left|\widehat{\psi}\left(k'\right)\right|^{2} \\ 
& \lesssim2^{2 (N + N') -2m\max\left(N,N'\right)}\cdot \left\Vert \psi\right\Vert _{L^{2}}^2 \leq 2^{-\left(2m-4\right)\max\left(N,N'\right)}\left\Vert \psi\right\Vert _{L^{2}}^2
\end{align*}
Picking $m$ so that $2m-4>\beta$ completes the proof of \eqref{eq:non-stationary}. 
\bigskip

\begin{proof}[Proof of Claim \ref{cla:suffFastPhaseAPP}]
    We distinguish three cases. 
    \begin{enumerate}
    \item \uline{The case \mbox{$\left(t,t'\right)=\left(u,u\right)$}}: we
    must have $N<N'+n\log_{2}\nu - \log_{2}\left(\frac{C_{0}}{4}\right)$.
    Then \eqref{eq:expansionFapp} implies 
    \begin{align*}
    \left|\nabla G\left(x\right)\right| & =\left|\left(D_{x}F^{n}\right)^{T}k'-k\right|\geq C_{0}\nu^{n}\left|k'\right|-\left|k\right|\\
     & \geq\frac{C_{0}^{-1}}{2}\nu^{n}2^{N'}-2^{N}\geq\max\left\{ \frac{C_{0}^{-1}}{4}\nu^{n}2^{N'},2^{N}\right\} \gtrsim_{\neg n}2^{\max\left(N,N'\right)}
    \end{align*}
    \item \uline{The case \mbox{$\left(t,t'\right)=\left(s,s\right)$}}: we
    must have $N'<N+n\log_{2}\nu-\log_{2}\left(\frac{C_{0}}{4}\right)$. Here, we bound 
    \begin{align*} | \nabla G(x)| &= \left| (D_x F^n)^{T} (k' - (D_x F^n)^{-T} k) \right| \gtrsim_n \left| k' - (D_x F^n)^{-T} k \right| \,  \\ 
        & \geq |(D_x F^n)^{-T} k| - |k'| \, . \end{align*}
    The contraction estimate \eqref{eq:contractionFapp} implies $|(D_x F^n)^{-T} k| \geq C_0^{-1} \nu^n |k|$ for $k \in \Cc^s_*$ (note $\Cc^s_*$ is unstable for $(D_x F^n)^{-T}$), and so
    \begin{align*}
    \left|\nabla G\left(x\right)\right| & \gtrsim_n 
     \frac12 2^N C_0^{-1} \nu^n - 2^{N'} \geq \frac14 C_0^{-1} \nu^n 2^N \gtrsim_{\neg n} 2^{\max(N, N')} 
    \end{align*}
    as desired. Note that in this case, the constant in front of the ultimate lower bound on $|\nabla G(x)|$ depends on $n$. 
    
    \item \uline{The case \mbox{$\left(t,t'\right)=\left(s,u\right)$}}: In this case, 
    $\left(D_{x}F^{n}\right)^{T}k'\in \mathcal{C}_*^{u}$ and $k\in \mathcal{C}_*^{s}$. In particular, \eqref{eq:angleBoundAPP} implies that the angle between $(D_x F^n)^T k'$ and $k$ is bounded uniformly from below, hence 
    \begin{align*} 
    \left|\left(D_{x}F^{n}\right)^{T}k'-k\right|& \gtrsim_{\neg n} \left|\left(D_{x}F^{n}\right)^{T}k'\right|+\left|k\right|\geq \frac12 C_0^{-1} \nu^n 2^{N'} + \frac{2^{N}}{2} \\ 
    & \gtrsim_{\neg n, \neg N, \neg N'}2^{\max\left(N,N'\right)}
    \end{align*}
    \end{enumerate}
    We conclude \eqref{eq:lower_grad_phi} holds in all cases, and the proof is complete. 
    \end{proof}

\subsection{Lasota-Yorke estimate for $\Lc_{f, \kappa} = e^{\kappa \Delta} \Lc_f $ for $\kappa > 0$}\label{subsec:treatDiffusivityAPP}

\newcommand{\uo}{{\underline{\omega}}}

For $\omega \in \R^2$, let $f_\omega : \T^2 \to \T^2$ be given by 
$f_\omega (x) = f (x) - \omega$ where the argument of $f$ appearing here is implicitly taken modulo 1. For $n \geq 1$ and $\uo = (\omega_1, \omega_2, \dots), \omega_i \in \R^2$, we write 
\[f^n_\uo = f_{\omega_n} \circ \dots \circ f_{\omega_1} \,. \]
Viewing the $(\omega_i)$ as an IID sequence of random variables on some probability space $(\Omega, \Fc, \P)$ allows us to recover the positive-diffusivity transfer operator $\Lc_{f, \kappa}$ from the `randomly perturbed' transfer operators
\[\Lc^n_\uo \varphi := \varphi \circ (f^n_\uo)^{-1} \,. \]
by Feynman-Kac as follows. 

\begin{lemma}[Stochastic realization]
    Assume the $\omega_i$ are IID, each with the (multivariate) normal distribution with mean zero and variance $2 \kappa I_2$. Then, for $\varphi \in C^\infty$ and for each $n \geq 1$, it holds that 
    \begin{align}\label{eq:stochRealizationAPP} \Lcfk^n \varphi(x) = \int \Lc^n_\uo \varphi(x) \, d \P(\uo) \end{align}
    for all $x \in \T^2$. 
\end{lemma}
This is a discrete-time version of the standard Feynman-Kac realization of the heat kernel by averages over sample paths of Brownian motion, the role of which is played here by the IID normal random variables $(\omega_i)$. { It can be proven by simply observing that for any $\phi \in C^\infty$:
\[
e^{\kappa\Delta}\phi\left(x\right)=\int \phi\left(x+\omega_{1}\right)\;d\mathbb{P}\left(\omega_{1}\right)
\] }
We note that the use of stochastic realization to control spectral properties of the ``averaged'' operators $\Lcfk$ is a standard technique in works on random dynamical systems; see, e.g., \cite[Proposition 2.3]{aimino2015annealed}.  

Crucially, the expansion and contraction estimates \eqref{eq:expansionFapp}, \eqref{eq:contractionFapp} for $F = f^{-1}$ hold as well for the compositions $F^n_\uo := (f^n_\uo)^{-1}$, since the noise here is additive and does not affect the desired cones relations in the derivatives $D_x F^n_\uo$ (or, for that matter, $C^\infty$ estimates on $f_\omega, F_\omega$). As a consequence, the work in Sections \ref{subsec:proofLYapp} and \ref{subsec:proofStatPhaseApp} applies just the same, mutatis mutandi, to almost-every realization $\uo$, yielding the following `random' version of the uniform Lasota-Yorke inequality: 
\begin{align} \label{eq:randomLYapp} \| \Lc^n_\uo \varphi \|_{\Hc^p} \leq C_2 \nu^{-n p} \| \varphi \|_{\Hc^p} + C_3 M^n \| \varphi \|_{H^{-L}} \qquad \P-\text{a.s.}\, ,  \end{align}
with constants $C_2, C_3$ and $M$ identical to those from the $\kappa = 0$ case. 

It remains to control $\| \Lcfk^n \varphi \|_{\Hc^p}$ in terms of $\| \Lc^n_\uo \varphi \|_{\Hc^p}$. For this, we will view $\uo \mapsto \Lc^n_\uo \varphi$ as an $\Hc^p$-valued random variable and note that \eqref{eq:stochRealizationAPP} implies that $\Lcfk^n \varphi$ coincides with the $\Hc^p$-Bochner integral $\int \Lc^n_\uo \varphi \, d \P(\uo)$. By the triangle inequality,
\[\| \Lcfk^n \varphi \|_{\Hc^p} = \left\| \int \Lc^n_\uo \varphi \, d \P(\uo) \right\|_{\Hc^p} \leq \int \| \Lc^n_\uo \varphi \|_{\Hc^p} d \P(\uo) \, ,  \]
the RHS of which is bounded by \eqref{eq:randomLYapp}.

This completes the proof of the uniform Lasota-Yorke estimate (Proposition \ref{prop:uniformLasotaYorke3}) in the case of $\kappa > 0$.

\printbibliography

\end{document}